\documentclass{amsart}
\usepackage{amsmath, amssymb, url}

\input xy
\xyoption{all}

\newtheorem{thm}{Theorem}[section]
\newtheorem{lem}[thm]{Lemma}
\newtheorem{cor}[thm]{Corollary}

\newtheorem{prop}[thm]{Proposition}

\theoremstyle{definition}
\newtheorem{example}[thm]{Example}
\newtheorem{rmk}[thm]{Remark}
\newtheorem{defn}[thm]{Definition}

\newtheorem{test}[thm]{Test}

\def\E{\mathcal{E}}
\def\K{\mathcal{K}}
\def\L{\mathcal{L}}
\def\M{\mathcal{M}}
\def\mN{\mathcal{N}}
\def\O{\mathcal{O}}
\def\N{\mathbb{N}}
\def\Q{\mathbb{Q}}
\def\Z{\mathbb{Z}}
\def\m{\mathfrak{m}}
\def\x{\mathbf{x}}
\def\y{\mathbf{y}}
\def\rv{\mathrm{rv}}
\renewcommand\vec[1]{\ensuremath\boldsymbol{#1}}

\begin{document}

\title{Multiplicative Valued Difference Fields}

\author{Koushik Pal}


\begin{abstract}
The theory of valued difference fields $(K, \sigma, v)$ depends on how the valuation $v$ interacts with the automorphism $\sigma$. Two special cases have already been worked out - the \textit{isometric} case, where $v(\sigma(x)) = v(x)$ for all $x\in K$, has been worked out by Luc Belair, Angus Macintyre and Thomas Scanlon \cite{BMS}; and the \textit{contractive} case, where $v(\sigma(x)) > n v(x)$ for all $n\in\mathbb{N}$ and $x\in K^\times$ with $v(x) > 0$, has been worked out by Salih Azgin \cite{A}. In this paper we deal with a more general version, called the \textit{multiplicative} case, where $v(\sigma(x)) = \rho\cdot v(x)$, where $\rho\; ( > 0)$ is interpreted as an element of a real-closed field. We give an axiomatization and prove a relative quantifier elimination theorem for such a theory.
\end{abstract}

\maketitle

\section{Introduction}
A \textit{valued field} is a structure $\mathcal{K} = (K, \Gamma, k; v, \pi)$, where $K$ is the underlying field, $\Gamma$ is an ordered abelian group (called the \textit{value group}), and $k$ is a field; $v: K\to\Gamma\cup\{\infty\}$ is the (surjective) \textit{valuation map}, with the \textit{valuation ring} (also called the \textit{ring of integers}) given by $$\O_K := \{a\in K: v(a)\ge 0\},$$ with a unique maximal ideal given by $$\m_K := \{a\in K: v(a) > 0\};$$ and $\pi: \mathcal{O}_K\to k$ is a surjective ring morphism. Then $\pi$ induces an isomorphism of fields $$a + \m_K\mapsto\pi(a) : \O_K/\m_K\to k,$$ and we identify the residue field $\O_K/\m_K$ with $k$ via this isomorphism. Accordingly $k$ is called the \textit{residue field}. When $K$ is clear from the context, we denote $\O_K$ and $\m_K$ by $\O$ and $\m$ respectively.

A \textit{valued difference field} is a valued field $\K$ as above with a distinguished automorphism (denoted by $\sigma$) of the base field $K$, which also satisfies $\sigma(\O_K) = \O_K$. It then follows that $\sigma$ induces an automorphism of the residue field: $$\pi(a)\mapsto \pi(\sigma(a)) : k \to k, \;\;\;\; a\in\O_K.$$ We denote this automorphism by $\bar{\sigma}$; and $k$  equipped with $\bar{\sigma}$ is called the \textit{residue difference field} of $\K$. Likewise, $\sigma$ induces an automorphism of the value group as well: $$\gamma\mapsto\sigma(\gamma) := v(\sigma(a)), \;\;\;\; \mbox{where } \gamma = v(a).$$ We denote this automorphism also by $\sigma$, and construe the value group as an ordered abelian group equipped with this special automorphism, and we call it the \textit{valued difference group}.

Depending on how the automorphism interacts with the valuation, we get different structures and hence different theories. For example, $\sigma$ is called \textit{isometric} if $v(\sigma(x)) = v(x)$ for all $x\in K$; and is called \textit{contractive} if $v(\sigma(x)) > nv(x)$ for all $n\in\N$, and $x\in K^{\times}$ with $v(x) > 0$. The existence of model companions of both these theories have been worked out in detail \cite{S}, \cite{BMS}, \cite{AD}, \cite{A}. There also has been a related work recently that needs mention. It is the work on valued ordered difference fields and rings by Fran\c coise Point \cite{P}. Note that in this theory the valued field itself is ordered, which, at least on the face of it, makes this theory quite different from the others mentioned.

Note that no matter how the automorphism interacts with the valuation, if we want any hope of having a model companion of the theory of a valued difference field, we better have a model companion of the theory of the valued difference group at least. But unfortunately, by Kikyo and Shelah's theorem \cite{KS}, the theory of a structure with the strict order property (e.g., an ordered abelian group) and a distinguished automorphism doesn't have a model companion. So we need to put some restriction on the automorphism. In the isometric case, $\sigma$ induces the identity automorphism on the value group; and so in this case, the value group is only an ordered abelian group, whose model companion is the theory of the ordered divisible abelian groups ($ODAG$). However, in the case when the induced automorphism is not the identity, the model companion (if it exists) should be able to decide how to extend the order between linear difference operators. In particular, for any $L(\gamma) = \sum_{i = 0}^{n} a_i \sigma^{i}(\gamma)$, where $a_i\in\Z$, $a_n\not= 0$ and $\gamma > 0$, the model companion should be able to decide when $L(\gamma) > 0$. In the contractive case, it is easily decided by the following condition: $$L(\gamma) > 0 \iff a_n > 0.$$

However, in more general cases, the decision criteria are not so simple. For example, it is not known whether the theory of an ordered abelian group $\Gamma$ with a strictly increasing automorphism ($\sigma(\gamma) > \gamma$ for all $0 < \gamma \in \Gamma$) has a model companion. So we restrict ourselves to a more specific case, where we impose that the induced automorphism $\sigma$ on the value group should satisfy the following axiom (scheme): for each $a_0, \ldots, a_n\in\Z$ and $L(\gamma) = \sum_{i = 0}^{n} a_i\sigma^i(\gamma)$,
$$\Big(\forall \gamma > 0 (L(\gamma) > 0)\Big)\bigvee\Big(\forall \gamma > 0 (L(\gamma) = 0)\Big)\bigvee\Big(\forall \gamma > 0 (L(\gamma) < 0)\Big)\;\;\;\;\;\;\;\;\;\;\;\;\;\;\;\;(\mbox{Axiom OM}).$$
It follows that for all $a, b\in\Z_+$,
$$\Big(\forall \gamma > 0 (a\sigma(\gamma) > b\gamma)\Big)\bigvee\Big(\forall \gamma > 0 (a\sigma(\gamma) = b\gamma)\Big)\bigvee\Big(\forall \gamma > 0 (a\sigma(\gamma) < b\gamma)\Big),$$
which is nothing but cuts with respect to rational multiples. Equivalently, we can also represent $\sigma$ as $$\sigma(\gamma) = \rho\cdot\gamma$$ for all $\gamma\in\Gamma$, where $\rho$ is interpreted as an element of an ordered ring, $\rho > 0$ and we make sense of the ``multiplication'' by defining the type of $\rho$ by Axiom OM. $\Gamma$ is then understood as an ordered module over that ordered ring. We call such a $\Gamma$ a \textit{multiplicative ordered difference abelian group} (henceforth, $MODAG$). We will show in Section 2 that the theory of such a $\Gamma$ has a model companion, the theory of multiplicative ordered divisible difference abelian group (henceforth, $MODDAG$).

In this paper we are thus interested in dealing with this more general case. We call $\sigma$ \textit{multiplicative} if $\sigma$ induces the structure of a $MODAG$ on $\Gamma$ via the rule 
$$v(\sigma(x)) = \rho. v(x)\;\;\;\;\mbox{ for all }x\in K,$$ 
where $\rho > 0$ (as interpreted in an ordered ring). The induced automorphism on the value group then satisfies $\sigma(\gamma) = \rho. \gamma$ for all $\gamma\in\Gamma$. $\rho$ is intended to be interpreted as an element of a real-closed field; for example, $\rho = 2$, or $\rho = \dfrac{5}{3}$, or $\rho = \sqrt{2}$, or $\rho = \pi$, or $\rho = 3 + \delta$ where $\delta$ is an infinitesimal, etc.

Three quick points should be noted here. First, we construe a $MODAG$ $\Gamma$ as an ordered $\Z[\rho, \rho^{-1}]$-module, where we think of $\rho\cdot\gamma$ as $\sigma(\gamma)$ and $\rho^{-1}\cdot\gamma$ as $\sigma^{-1}(\gamma)$. Clearly, $\rho^m\cdot\gamma = \sigma^m(\gamma)$ for all $m\in\Z$ and all $\gamma\in\Gamma$. To be able to extend $\Gamma$ to a model of $MODDAG$, we would then want divisibility by ``non-zero'' linear difference operators, which typically look like $L = \sum_{l = 1}^{n} a_l\rho^l + \sum_{l = 1}^{m} b_l\rho^{-l}$, with $a_n \not= 0 \not= b_m$. Any question of solvability of a system $L\cdot x = b$ for $b\in\Gamma$, can then easily be transformed to a question involving only $\rho$, by multiplying the equation throughout by $\rho^m$. Thus, $(\sum_{l = 1}^{n} a_l\rho^l + \sum_{l = 1}^{m} b_l\rho^{-l})\cdot x = b$ is solvable if and only if $(\sum_{l = 1}^{n} a_l\rho^{m + l} + \sum_{l = 1}^{m} b_l\rho^{m-l})\cdot x = \rho^m(b)$ is solvable. In particular, for all practical purposes we can think of $\Gamma$ as a $\Z[\rho]$-module, with the understanding that $\rho$ has an inverse.

The second point to be noted is that if $\sigma(\gamma) = \rho\cdot\gamma$, then $\sigma^{-1}(\gamma) = \rho^{-1}\cdot\gamma$. In particular, if $0 < \rho \le 1$, we can shift to $\sigma^{-1}$, and instead work with $\rho^{-1} \ge 1$. Thus, without loss of generality, we may assume that $\rho\ge 1$.

And finally, the third point to be noted is that this is a generalization over the isometric and the contractive cases. The case $\rho = 1$ is precisely the isometric case; and the case ``$\rho = \infty$'', i.e., when all $0<\gamma\in\Gamma$ satisfy for all $b\in\Z_+$, $\rho\cdot\gamma > b\gamma$, is the contractive case. We can have other finite and infinitesimal values for $\rho$ as well.

Also one more thing needs mention here about the characteristics of the relevant fields. Any automorphism of a field is trivial on the integers. Thus for any $n\in\Z$, we have $\sigma(n) = n$. In particular, this means that for any prime $p$, if $v(p) > 0$, then $v(p) = v(\sigma(p)) = \rho. v(p)$, which implies $\rho=1$. Thus the mixed characteristic case doesn't arise for $\rho > 1$, and the mixed characteristic case for $\rho = 1$ has already been dealt with in \cite{BMS}. The equi-characteristic $p$ case even without the automorphism is too non-trivial and is not known yet. So we will restrict ourselves only to the equi-characteristic zero case in this paper.

\bigskip
\section{Multiplicative Ordered Difference Abelian Group ($MODAG$)}
We work in the language of ordered groups with a symbol for an automorphism and its inverse $\L_{\rho\cdot, <} = \{+, -, 0, <, \rho\cdot, \rho^{-1}\cdot\}$.
\newline

The $\L_{\rho\cdot, <}$-theory $T_{\rho\cdot, <}$ of ordered difference abelian groups can be axiomatized by the following axioms:
\begin{enumerate}
\item{Axioms of Abelian Groups in the language \{+, -, 0\}}
\item{Axioms of Linear Order in the language \{$<$\}}
\item{\begin{itemize}
\item{$\forall x\forall y\forall z(x < y\implies x + z < y + z)$}
\item{$\forall x\forall y(x < y\implies \rho\cdot x < \rho\cdot y)$}
\item{$\forall x\forall y(x < y\implies \rho^{-1}\cdot x < \rho^{-1}\cdot y)$}
\end{itemize}}
\item{Axioms mentioning $\rho\cdot$ and $\rho^{-1}\cdot$ are endomorphisms, and they are inverses of each other - thus making $\rho$ an automorphism.}
\end{enumerate}
Note that $T_{\rho\cdot, <}$ is an universal $\L_{\rho\cdot, <}$-theory.
\newline

Unfortunately the theory of ordered abelian groups has the strict order property. And hence, by Kikyo and Shelah's theorem \cite{KS}, we cannot hope to have a model companion of the theory of ordered difference abelian groups.

However, if we restrict ourselves to very specific kind of automorphisms, we do actually get model companion. The intended automorphisms are multiplication by an element of a real-closed field, for example, $\rho\cdot x = 2x$, or $\rho\cdot x = \sqrt{2}x$, or $\rho\cdot x = \delta x$, where $\delta$ could be an infinite or infinitesimal element.

The problem, however, is that in general abelian groups such multiplications do not make sense. But since integers embed in any abelian group, by imitating what we do for real numbers, we can make sense of such multiplications.
\newline
Note that for an abelian group $G$, multiplication by $\N$ makes sense: $mg := \overbrace{g + \cdots + g}^{m\mbox{ times}}.$ 
\newline
Taking additive inverses, multiplication by $\Z$ also makes sense: $(-m)g := -(mg).$
\newline
If $G$ is torsion-free divisible, multiplication by $\Q$ makes sense: $\dfrac{m}{n}g = \dfrac{mg}{n} := $ the unique $y$ such that $ny = mg.$

We carry this idea forward and define cuts with rational numbers to make sense of multiplication by irrationals. And in the process we get multiplication by infinite numbers and infinitesimals as well. Unfortunately this idea doesn't quite work and so we need a little stronger axiom as we will see below.

Our intended models are the additive groups of ordered $\Z[\rho, \rho^{-1}]$-modules. For $i\in\N$, we denote
\begin{center}
$\rho^{i}\cdot x := \overbrace{\rho\cdot\rho\cdot\ldots\cdot\rho\cdot}^{i\mbox{ times}}x$ and $\rho^{-i}\cdot x := \overbrace{\rho^{-1}\cdot\rho^{-1}\cdot\ldots\cdot\rho^{-1}\cdot}^{i\mbox{ times}}x$.
\end{center}
As noted in the introduction, for all practical purposes, we can restrict ourselves only to ordered $\Z[\rho]$-modules. There is a natural map $\Phi: \Z[\rho]\to End(G)$, which maps any $L := m_k \rho^{k} + m_{k - 1}\rho^{k - 1} + \ldots + m_1 \rho + m_0$ (thought of as an element of $\Z[\rho]$ with the $m_i$'s coming from $\Z$), to an endomorphism $L: G\to G$. Such an $L$ is called a linear difference operator. And we make sense of this multiplication by imposing the following additional condition on $\rho$: for each $L\in\Z[\rho]$,
$$\Big(\forall x > 0\; (L\cdot x > 0)\Big)\bigvee\Big(\forall x > 0\; (L\cdot x = 0)\Big)\bigvee\Big(\forall x> 0\; (L\cdot x < 0)\Big)\;\;\;\;\;\;\;\;\;\;\;\;\;(\mbox{Axiom OM})$$
(OM stands for Ordered Module). It immediately follows that for each $a, b\in\Z_+$,
$$\Big(\forall x > 0\; (a\rho\cdot x > bx)\Big)\bigvee\Big(\forall x > 0\; (a\rho\cdot x = bx)\Big)\bigvee\Big(\forall x> 0\; (a\rho\cdot x < bx)\Big)$$
which is nothing but cuts with respect to rational multiples (recall that since we are considering only order-preserving automorphisms, $\rho > 0$). For any $\rho$ satisfying Axiom OM, we also define the order type of $\rho$ (relative to $\Gamma$) as
$$otp_{\Gamma}(\rho) := \{L\in\Z[\rho] : \forall x \in\Gamma\; (x>0\implies L\cdot x > 0)\},$$
and we say two $\rho$ and $\rho'$ are same if they have the same order type.

Note that Axiom OM is consistent with Axioms 1-4 because any ordered abelian group is a model of these axioms with $\rho = 1$. Also note that with this axiom $\Z[\rho]$ becomes an ordered commutative ring: $L_1 \geqq L_2$ iff $\forall x > 0 \Big((L_1 - L_2)\cdot x \geqq 0\Big)$.
\begin{defn}
An ordered difference abelian group is called \textit{multiplicative} if it satisfies Axiom OM. The theory of such structures (called as $MODAG$) is axiomatized by Axioms 1-4 and Axiom OM. Note that this theory is also universal.

We also denote by $MODAG_\rho$ the theory $MODAG$ where the order type of $\rho$ is fixed.
\end{defn}

If there is a non-zero $L\in\Z[\rho]$ such that $\forall x > 0 (L\cdot x = 0)$, we say $\rho$ is algebraic (over the integers); otherwise we say $\rho$ is transcendental. If $\rho$ is algebraic, there is a minimal (degree) polynomial that it satisfies.

Note that the kernel of $\Phi$ need not be trivial. For example, if $\rho\cdot x = 2x$ for all $x$, then $\rho - 2\in $ Ker$(\Phi)$. In particular, Ker$(\Phi)$ is non-trivial iff $\rho$ is algebraic. We then form the following ring:
$$\widetilde{\Z[\rho]} := \Z[\rho]/Ker(\Phi).$$

\begin{defn}
A difference group $G$ is called \textit{divisible} (or \textit{linear difference closed}) if for any non-zero $L \in \widetilde{\Z[\rho]}$ and $b\in G$, the system $L\cdot x = b$ has a solution in $G$.
\end{defn}

\begin{defn}
Let $MODDAG$ be the $\L_{\rho, <}$-theory of non-trivial multiplicative ordered divisible difference abelian groups. This theory is axiomatized by the above axioms along with 
$$\exists x (x\not=0)$$
and the following additional infinite list of axioms: for each $L\in\Z[\rho]$,
$$\Big(\forall\gamma\in\Gamma (L\cdot\gamma = 0) \Big) \vee \Big(\forall\gamma\in\Gamma\;\exists\delta\in\Gamma (L\cdot\delta = \gamma)\Big),$$
i.e., all non-zero linear difference operators are surjective. Thus, $MODDAG$ is an $\forall\exists$-theory. Similarly as above, we denote by $MODDAG_\rho$ the theory $MODDAG$ where we fix the order type of $\rho$.
\end{defn}

\vspace{1em}
We would now like to show that $MODDAG$ is the model companion of $MODAG$. By abuse of terminology, we would refer to any model of $MODAG$ (respectively $MODDAG$) also as $MODAG$ (respectively $MODDAG$).
\newline\newline
\textbf{Remark.} It might already be clear from the definitions above that for a given $\rho$, $MODDAG_\rho$ is basically the theory of non-trivial ordered vector spaces over the ordered field $\Q(\rho)$ and then quantifier elimination actually follows from well-known results. However, here we are doing things a little differently. Instead of proving the result for a particular $\rho$, we are proving it uniformly across all $\rho$ using Axiom OM. And even though in the completion the type of $\rho$ is determined and the theory actually reduces to the above well-known theory, nevertheless it makes sense to write down some of the trivial details just to make sure that nothing fishy happens.
\newline

\begin{lem}
$MODAG$ and $MODDAG$ are co-theories. 
\end{lem}
\begin{proof}
We will actually prove something stronger: for a fixed $\rho$, $MODAG_\rho$ and $MODDAG_\rho$ are co-theories. Any model of $MODDAG_\rho$ is trivially a model of $MODAG_\rho$. So all we need to show is that we can embed any model $G$ of $MODAG_\rho$ into a model of $MODDAG_\rho$.

If $G$ is trivial, we can embed it into $\Q$ with any given $\rho$. So without loss of generality we may assume, $G$ is non-trivial.

Let $\widetilde{\Z[\rho]}_+ := \{L\in\widetilde{\Z[\rho]} : L > 0\}$. Define an equivalence relation $\sim$ on $G\times \widetilde{\Z[\rho]}_+$ as follows:
$$(g, L) \sim (g', L') \iff L'\cdot g = L\cdot g'.$$
Reflexivity and symmetry are obvious. For transitivity, suppose $$(g, L)\sim (g', L')\mbox{ and } (g', L')\sim (g'', L'').$$ 
Then, $$L'\cdot g = L\cdot g' \mbox{ and } L'\cdot g'' = L''\cdot g'.$$
Applying $L''$ to the first equation and $L$ to the second equation, we get $L''\cdot L'\cdot g = L''\cdot L\cdot g'$ and $L\cdot L'\cdot g'' = L\cdot L''\cdot g'$. Since these operators commute with each other, we can rewrite this as:
$$L'\cdot L''\cdot g = L''\cdot L'\cdot g = L''\cdot L\cdot g' = L\cdot L''\cdot g' = L\cdot L'\cdot g'' = L'\cdot L\cdot g'',$$
i.e., $L'\cdot (L''\cdot g - L\cdot g'') = 0$. Since $L'\in\widetilde{\Z[\rho]}_+$, i.e., $L' > 0$, we have $L''\cdot g = L\cdot g''$, i.e., $(g, L)\sim (g'', L'')$.

Let $[(g, L)]$ denote the equivalence class of $(g, L)$ and let $H = G\times \widetilde{\Z[\rho]}_+/\sim$. 

We define $+$ on $H$ by $[(g, L)] + [(h, P)] = [P\cdot g + L\cdot h, L\cdot P]$, where by $L\cdot P$ we mean $L\circ P$.

To show that this is well-defined, let $(g, L)\sim (g', L')$. Want to show that $[(g, L)] + [(h, P)] = [(g', L')] + [(h, P)]$, i.e., $[(P\cdot g + L\cdot h, L\cdot P)] = [(P\cdot g' + L'\cdot h, L'\cdot P)]$. In other words, we want to show that
$$(P\cdot g + L\cdot h, L\cdot P) \sim (P\cdot g' + L'\cdot h, L'\cdot P).$$
But,
$$L\cdot (P\cdot (P\cdot g' + L'\cdot h)) = L\cdot P\cdot P\cdot g' + L\cdot P\cdot L'\cdot h = P\cdot P\cdot L\cdot g' + L'\cdot P\cdot L\cdot h$$
$$ = P\cdot P\cdot L'\cdot g + L'\cdot P\cdot L\cdot h = L'\cdot P\cdot P\cdot g + L'\cdot P\cdot L\cdot h = L'\cdot (P\cdot (P\cdot g + L\cdot h)).$$
Hence, $+$ is well-defined. Similarly, we can define $-$ by $$[(g, L)] - [(h, P)] = [(P\cdot g - L\cdot h, L\cdot P)].$$ 
This is also well-defined. It follows easily that $(H, +, -)$ is an abelian group, where $[(0, 1)]$ is the identity and $[(-g, L)]$ is the inverse of $[(g, L)]$.
\newline

We define an automorphism of $H$ (which we still denote by $\rho\cdot$) as follows: let $\rho\cdot [(g, L)] = [(\rho\cdot g, L)]$. It is easy to check that this is well-defined and defines an automorphism of $H$. \newline

For any non-zero $L\in \Z[\rho]$ and any $[(h, P)]\in H$, we have
$$L\cdot [(h, P\cdot L)] = [(L\cdot h, P\cdot L)] = [(h, P)].$$
Hence, $H$ is linear difference closed or divisible.
\newline

We extend the order as follows:
$$[(g, L)] < [(g', L')] \iff L'\cdot g < L\cdot g'.$$
If $g, h\in G$ with $g < h$, then $[(g, 1)] < [(h, 1)]$; so this extends the ordering on $G$. Moreover, for $[(a_1, L_1)] < [(a_2, L_2)]$ and $[(b_1, P_1)] \le [(b_2, P_2)]$, we have $L_2\cdot a_1 < L_1\cdot a_2$ and $P_2\cdot b_1 \le P_1\cdot b_2$. Then,
\begin{eqnarray*}
P_1 \cdot P_2\cdot  L_2\cdot a_1 + L_1\cdot L_2\cdot P_2\cdot b_1 & < & P_1 \cdot P_2\cdot L_1\cdot a_2 + L_1\cdot L_2\cdot P_1\cdot b_2 \\
\mbox{i.e., } L_2 \cdot P_2\cdot (P_1\cdot a_1 + L_1\cdot b_1) & < & L_1\cdot P_1\cdot (P_2\cdot a_2 + L_2\cdot b_2) \\
\mbox{i.e., } [(P_1\cdot a_1 + L_1\cdot b_1, L_1\cdot P_1)] & < & [(P_2\cdot a_2 + L_2\cdot b_2, L_2\cdot P_2)]  
\end{eqnarray*}
Also, 
$$[(a, L)] < [(b, P)]\iff P\cdot a < L\cdot b \iff \rho\cdot P\cdot  a < \rho\cdot L\cdot b \iff P\cdot \rho\cdot a < L\cdot \rho\cdot b \iff [(\rho\cdot a, L)] < [(\rho\cdot b, P)].$$

Finally, for any $L\in\Z[\rho]$, and $x > 0$ and $P\in\widetilde{\Z[\rho]}_+$, we have
$$L\cdot [(x, P)] > [(0, 1)] \iff [L\cdot x, P] > [(0, 1)] \iff L\cdot x > 0.$$
Hence, $H$ is a multiplicative ordered divisible difference abelian group.
\newline\newline
\textbf{Claim.} $G$ embeds into $H$.
\newline
\textbf{Proof.} Define $\iota: G\to H$ by $\iota(g) = [(g, 1)].$ Then
\newline
$\iota(0) = [(0, 1)]; \iota(g + h) = [(g + h , 1)] = [(g, 1)] + [(h, 1)]; \iota(-g) = [(-g, 1)] = -[(g, 1)].$
\newline
Also, $\iota(\rho\cdot g) = [(\rho\cdot g, 1)] = \rho\cdot [(g, 1)]$.
\newline
And, $g < h \implies \iota(g) = [(g, 1)] < [(h, 1)] = \iota(h)$.

Moreover, if $H'\models MODDAG_\rho$ and $j: G\to H'$ is an embedding, then let $h: H\to H'$ be given by $h([(g, L)]) = [(j(g), L)]$. It is routine to check that $h$ is a well-defined embedding, preserves order and $j = h\circ\iota$.
We, thus, call $H$ as the (multiplicative) divisible hull of $G$.
\end{proof}

We have thus shown that for a fixed $\rho$, $MODAG_\rho$ and $MODDAG_\rho$ are co-theories. In fact, since $(MODDAG_\rho)_\forall = MODAG_\rho$, what we have actually shown is that $MODDAG_\rho$ has algebraically prime models, namely the (multiplicative) dvisible hull. We will now show that $MODDAG_\rho$ eliminates quantifiers.
\begin{lem}
$MODDAG_\rho$ has quantifier elimination.
\end{lem}
\begin{proof}
The relevant $\rho$'s correspond bijectively to the numbers $1, 1+ \epsilon, a-\epsilon, b, a+\epsilon, 1/\epsilon$, where we fix some positive infinitesimal $\epsilon$ in an ordered field extension of the field $\mathbb{R}$ of real numbers, and $a$ ranges over the real algebraic numbers $>1$, and $b$ over the real numbers $>1$. For each such $\rho$, we have the ordered field $\Q(\rho)$. The construction in Lemma 2.4 essentially shows that $MODDAG_\rho$ is the theory of non-trivial ordered vector spaces over the ordered field $\Q(\rho)$, which admits quantifier elimination by well-known standard results, see \cite{D}.
\end{proof}

Thus, $MODDAG_\rho$ eliminates quatifiers. In particular, $MODDAG_\rho$ is model complete. Moreover, for a fixed $\rho$, $\Q(\rho)$ with the induced ordering is a prime model of $MODDAG_\rho$. In particular, $MODDAG_\rho$ is complete. Note that $MODDAG$ is not complete; its completions are given by $MODDAG_\rho$ by fixing a (consistent) order type of $\rho$. Finally we have,

\begin{thm}
$MODDAG$ is the model companion of $MODAG$.
\end{thm} 
\begin{proof}
By Lemma 2.4, $MODAG$ and $MODDAG$ are co-theories. All we need to show now is that $MODDAG$ is model complete.

So let $G\subseteq H$ be two models of $MODDAG$. Want to show that $G\preceq H$.

Since $G\subseteq H$ and both are non-trivial, in particular they have the same order type of $\rho$. Thus, for some fixed $\rho$, $G, H\models MODDAG_\rho$. But $MODDAG_\rho$ is model complete. 

Hence, $G\preceq H$.
\end{proof}

\bigskip
\section{Preliminaries}
Let $\K\prec\K'$ be an extension of valued fields. For any $a\in\K'$, $K\langle a\rangle$ denotes the smallest difference subfield of $\K'$ containing $K$ and $a$. The underlying field of $K\langle a\rangle$ is $K(\sigma^i(a) : i\in\Z)$. In literature a difference field generally means a field with an endomorphism. For our case, a difference field always means a field with an automorphism. So ``the smallest difference subfield'' in our context actually means the smallest inversive difference subfield.

For any $(n+1)$-variable polynomial $P(X_0, \ldots, X_n)\in K[X_0, \ldots, X_n]$, we define a corresponding $1$-variable $\sigma$-polynomial $f(x) = P(x, \sigma(x), \sigma^2(x), \ldots, \sigma^n(x)).$ We define the \textit{degree} of $f$ to be the total degree of $P$; and the \textit{order of $f$} to be the largest integer $0\le d\le n$ such that the coefficient of $\sigma^d(x)$ in $f(x)$ is non-zero. If $f\in K$, then order$(f) := -\infty$. Finally we define the complexity of $f$ as $$\mbox{complexity}(f) := (d, \mbox{deg }_{x_d} f, \mbox{deg } f) \in (\N\cup\{-\infty\})^3,$$
where complexity$(0) := (-\infty, -\infty, -\infty)$ and for $f\in K, f\not= 0$, complexity$(f) := (-\infty, 0, 0)$. We order complexities lexicographically.

Let $\x = (x_0, \ldots, x_n)$, $\y = (y_0, \ldots, y_n)$ be tuples of indeterminates and $\vec{a} = (a_0, \ldots, a_n)$ be a tuple of elements from some field. Let $\vec{I} = (i_0, \ldots, i_n)$ be a multi-index ($\vec{I}\in\Z^{n+1}$). We define the \textit{length} of $\vec{I}$ as $|\vec{I}| := i_0 + \cdots + i_n$ and $\vec{a}^{\vec{I}} := a_0^{i_0}\cdots a_n^{i_n}$. For any element $\rho$ of any ring, we define the $\rho$-length of $\vec{I}$ as $|\vec{I}|_\rho := i_0\rho^0 + i_1\rho^1 + \cdots + i_n\rho^n.$ Then $|\vec{I}|\in\Z$ and $|\vec{I}|_\rho$ is an element of that ring. For any polynomial $P(\x)$ over $K$, we have a unique Taylor expansion in $K[\x, \y]:$
$$P(\x + \y) = \sum_{\vec{I}} P_{(\vec{I})}(\x)\cdot \y^{\vec{I}},$$
where the sum is over all $\vec{I} = (i_0, \ldots, i_n)\in\N^{n+1}$, each $P_{(\vec{I})}(x)\in K[\x]$, with $P_{(\vec{I})} = 0$ for $|\vec{I}| > $ deg$(P)$, and $\y^{\vec{I}} := y_0^{i_0} \cdots y_n^{i_n}$. Thus $\vec{I}!P_{(\vec{I})} = \mathbf{\partial}_{\vec{I}}P$ where $\mathbf{\partial}_{\vec{I}}$ is the operator $(\partial/\partial x_0)^{i_0}\cdots(\partial/\partial x_n)^{i_n}$ on $K[\x]$, and $\vec{I}! := i_0! \cdots i_n!.$ We construe $\N^{n + 1}$ as a monoid under $+$ (componentwise addition), and let $\le$  be the (partial) product ordering on $\N^{n + 1}$ induced by the natural order on $\N$. Define for $\vec{I}\le\vec{J}\in\N^{n + 1}$,
$$\Big(\begin{tabular}{c}$\vec{J}$ \\ $\vec{I}$\end{tabular}\Big) := \Big(\begin{tabular}{c} $j_0$ \\ $i_0$\end{tabular}\Big) \cdots \Big(\begin{tabular}{c} $j_n$ \\ $i_n$\end{tabular}\Big).$$
Then it is easy to check that for $\vec{I}, \vec{J}\in\N^{n + 1}$,  
$$(f_{(\vec{I})})_{(\vec{J})} = \Big(\begin{tabular}{c}$\vec{I + J}$ \\ $\vec{I}$\end{tabular}\Big)f_{(\vec{I} + \vec{J})}.$$
Let $x$ be an indeterminate. When $n$ is clear from the context, we set $\vec{\sigma}(x) := (x, \sigma(x), \ldots, \sigma^n(x))$, and also $\vec{\sigma}(a) = (a, \sigma(a), \ldots, \sigma^n(a))$ for $a\in K$. Then for $P\in K[x_0, \ldots, x_n]$ as above and $f(x) = P(\vec{\sigma}(x))$, we have
\begin{eqnarray*}
f(x + y) & = & P(\vec{\sigma}(x + y)) = P(\vec{\sigma}(x) + \vec{\sigma}(y)) \\
& = & \sum_{\vec{I}} P_{(\vec{I})}(\vec{\sigma}(x)) \cdot \vec{\sigma}(y)^{\vec{I}} = \sum_{\vec{I}} f_{(\vec{I})}(x)\cdot \vec{\sigma}(y)^{\vec{I}},
\end{eqnarray*}
where $f_{(\vec{I})}(x) := P_{(\vec{I})}(\vec{\sigma}(x)).$

A \textit{pseudo-convergent sequence} (henceforth, \textit{pc-sequence}) from $K$ is a limit ordinal indexed sequence $\{a_\eta\}_{\eta < \lambda}$ of elements of $K$ such that for some index $\eta_0$, 
$$\eta'' > \eta' > \eta \ge \eta_0 \implies v(a_{\eta''} - a_{\eta'}) > v(a_{\eta'} - a_{\eta}).$$

We say $a$ is a \textit{pseudo-limit} of a limit ordinal indexed sequence $\{a_\eta\}$ from $K$ (denoted $a_\eta\leadsto a$) if there is some index $\eta_0$ such that
$$\eta' > \eta \ge \eta_0 \implies v(a - a_{\eta'}) > v(a - a_{\eta}).$$

Note that such a sequence is necessarily a pc-sequence in K. For a pc-sequence $\{a_\eta\}$ as above, let $\gamma_\eta := v(a_{\eta'} - a_\eta)$ for $\eta' > \eta \ge \eta_0$; note that this depends only on $\eta$. Then $\{\gamma_\eta\}_{\eta \ge \eta_0}$ is strictly increasing. We define the \textit{width} of $\{a_\eta\}$ as the set $$\{\gamma\in\Gamma\cup\{\infty\}: \gamma > \gamma_\eta \mbox{ for all } \eta\ge\eta_0\}.$$

We say two pc-sequences $\{a_\eta\}$ and $\{b_\eta\}$ from $K$ are \textit{equivalent} if they have the same pseudo-limits in all valued field extensions of $\mathcal{K}$. Equivalently, $\{a_\eta\}$ and $\{b_\eta\}$ are equivalent iff they have the same width and a common pseudo-limit in some extension of $\mathcal{K}$.

\bigskip
\section{Pseudoconvergence and Pseudocontinuity}
As already stated, we are interested in proving an Ax-Kochen-Ershov type theorem for (and hence, finding the model companion of) the theory of multiplicative valued difference fields (valued fields where $\sigma$ is multiplicative). In this paper, we are always in equi-characteristic zero. So all valued fields and residue fields are of characteristic zero. Our main axiom is
$$\mathbf{Axiom\; 1. }\;\;\;\;\; v(\sigma(x)) = \rho\cdot v(x) \;\;\;\;\forall x\in K \mbox{ and } \rho\ge 1.$$
The value group $\Gamma\models MODAG$, and as already mentioned before, such a multiplication makes sense in a $MODAG$, where $\Z[\rho]$ is an ordered ring and $\Gamma$ is construed as an ordered module over that ring. From now on, we assume that \textit{all our valued difference fields and valued difference field extensions satisfy Axiom 1.} 

Our first goal is to prove pseudo-continuity. It follows from \cite{K} that if $\{a_\eta\}$ is a pc-sequence from $K$, and $a_\eta\leadsto a$ with $a\in K$, then for any ordinary non-constant polynomial $P(x) \in K[x]$, we have $P(a_\eta)\leadsto P(a)$. Unfortunately, this is not true in general for non-constant $\sigma$-polynomials over valued difference fields. As it turns out, this is true when $\rho$ is transcendental over the integers (which includes the contractive case ``$\rho = \infty$''), but not true if $\rho$ is algebraic (which includes the isometric case $\rho = 1$). Fortunately, in the algebraic case, we can remedy the situation by resorting to equivalent pc-sequences. We will follow the treatment of \cite{BMS}, \cite{AD} with appropriate modifications. We will, however, need the following basic lemma.
\begin{lem}
Let $\{\gamma_\eta\}$ be an increasing sequence of elements in a $MODAG$ $\Gamma$. Let $A = \{|I_i|_\rho : I_i\in\Z^{n + 1}, i = 1, \ldots, l\}$ be a finite set with $|A| = m$, and for $i = 1, \ldots, m$, let $c_i + n_i\cdot x$, $c_i\in\Gamma$, $n_i \in A$, be linear functions of $x$ with distinct $n_i$. Then there is a $\mu$, and an enumeration $i_1, i_2, \ldots, i_m$ of $\{1, \ldots, m\}$ such that for $\eta > \mu$, $c_{i_1} + n_{i_1}\cdot\gamma_\eta < c_{i_2} + n_{i_2}\cdot\gamma_\eta < \cdots < c_{i_m} + n_{i_m}\cdot\gamma_\eta$.
\end{lem}
\begin{proof}
Since $\Gamma$ is a $MODAG$, there is a linear order amongst the $n_i$'s. Suppose $n_i\not=n_j\in A$. WMA $n_i < n_j$. Then either $c_j + n_j\cdot\gamma_\eta < c_i + n_i\cdot\gamma_\eta$ for all $\eta$, or for some $\eta_{ij}$, $c_i + n_i\cdot\gamma_{\eta_{ij}} \le c_j + n_j\cdot\gamma_{\eta_{ij}}$. But in the later case, for all $\eta > \eta_{ij}$, we have $c_i + n_i\cdot\gamma_\eta < c_j + n_j\cdot\gamma_\eta$, as $n_i < n_j$ and $\{\gamma_\eta\}$ is increasing. Since $A$ is a finite set, the set of all such $\eta_{ij}$'s is also finite, and hence taking $\mu$ to be the maximum of those $\eta_{ij}$'s, we have our result.
\end{proof}

\hspace{-1.2em}\textbf{\underline{Basic Calculation.}}
\newline\newline
Suppose $\K$ is a multiplicative valued difference field. Let $\{a_\eta\}$ be a pc-sequence from $K$ with a pseudo-limit $a$ in some extension. Let $P(x)$ be a non-constant $\sigma$-polynomial over $K$ of order $\le n$.
\newline\newline
\textbf{Case I. $\rho$ is transcendental.}

Let $\gamma_\eta = v(a_\eta - a)$. Then for each $\eta$ we have,
\begin{eqnarray*}
P(a_\eta) - P(a) & = & \sum_{\substack{\vec{L}\in\N^{n+1} \\ 1\le |\vec{L}| \le \deg(P)}} P_{(\vec{L})}(a)\cdot \vec{\sigma}(a_\eta - a)^{\vec{L}} =: \sum_{\substack{\vec{L}\in\N^{n+1} \\ 1\le |\vec{L}| \le \deg(P)}} Q_{\vec{L}}(\eta)
\end{eqnarray*}

To calculate $v(P(a_\eta) - P(a))$, we need to calculate the valuation of each summand $Q_{\vec{L}}(\eta)$. We claim that there is a unique $\vec{L}$ for which the valuation of $Q_{\vec{L}}(\eta)$ is minimum eventually. 
Suppose not. Note that the valuation of $Q_{\vec{L}}(\eta)$
$$v(Q_{\vec{L}}(\eta)) = v(P_{(\vec{L})}(a)\cdot\vec{\sigma}(a_\eta - a)^{\vec{L}}) = v(P_{(\vec{L})}(a)) + |\vec{L}|_\rho\cdot\gamma_\eta$$
is a linear function in $\gamma_\eta$. Thus, by Lemma 4.1, the only way there isn't a unique $\vec{L}$ with the valuation of $Q_{\vec{L}}(\eta)$ minimum eventually is if there are $\vec{L}\not=\vec{L'}$ with $|\vec{L}|_\rho = |\vec{L'}|_\rho$. But then,
\begin{eqnarray*}
|\vec{L}|_\rho = |\vec{L'}|_\rho & \implies & |\vec{L - L'}|_\rho = 0 \\
& \implies & (l_0 - l'_0)\rho^0 + (l_1 - l'_1)\rho^1 + \cdots + (l_n - l'_n)\rho^n = 0
\end{eqnarray*}
which implies that $\rho$ is algebraic over $\Z$, a contradiction. Hence, the claim holds. In particular, there is a unique $\vec{L}_0$ such that eventually (in $\eta$),
$$v(P(a_\eta) - P(a)) = v(P_{(\vec{L}_0)}(a)) + |\vec{L_0}|_\rho\cdot\gamma_\eta,$$
which is strictly increasing. Hence, $P(a_\eta)\leadsto P(a)$.

Note that if $\rho = \infty$, then $\rho$ is transcendental over $\Z$. Hence, the contractive case is included in Case I.
\newline\newline
\textbf{Case II. $\rho$ is algebraic.}

Since $\rho$ satisfies some algebraic equation over the integers, there can be accidental cancelations and we might have $v(Q_{\vec{L}}(\eta)) = v(Q_{\vec{L'}}(\eta))$ for infinitely many $\eta$ and $\vec{L}\not=\vec{L'}$, and the above proof fails. To remedy this, we construct an equivalent pc-sequence $\{b_\eta\}$ such that $P(b_\eta)\leadsto P(a)$. 

Put $\gamma_\eta := v(a_\eta - a)$; then $\{\gamma_\eta\}$ is eventually strictly increasing.  Since $v$ is surjective, choose $\theta_\eta\in K$ such that $v(\theta_\eta) = \gamma_\eta$. Set $b_\eta := a_\eta + \mu_\eta \theta_\eta$, where we demand that $\mu_\eta\in K$ and $v(\mu_\eta) = 0$. Define $d_\eta$ by $a_\eta - a = \theta_\eta d_\eta$. So $v(d_\eta) = 0$ and $d_\eta$ depends on the choice of $\theta_\eta$. Since $a$ is normally not in $K$, $d_\eta$ won't normally be in $K$ either. Then,
\begin{eqnarray*}
b_\eta - a & = & b_\eta - a_\eta + a_\eta - a \\
& = & \theta_\eta(\mu_\eta + d_\eta). 
\end{eqnarray*}

We impose $v(\mu_\eta + d_\eta) = 0$. This ensures $b_\eta\leadsto a$, and that $\{a_\eta\}$ and $\{b_\eta\}$ have the same width; so they are equivalent. Let $A := \{|\vec{L}|_\rho : \vec{L}\in\mathbb{N}^{n+1} \mbox{ and } 1\le |\vec{L}| \le \deg(P)\}$. Now,
\begin{eqnarray*}
P(b_\eta) - P(a) & = & \sum_{|\vec{L}|_\rho\; \in\; A} P_{(\vec{L})}(a)\cdot\vec{\sigma}(b_\eta - a)^{\vec{L}} \\
& = & \sum_{m \in A} \sum_{|\vec{L}|_\rho = m} P_{(\vec{L})}(a) \cdot\vec{\sigma}(b_\eta - a)^{\vec{L}} \\
& = & \sum_{m \in A} \sum_{|\vec{L}|_\rho = m} P_{(\vec{L})}(a) \cdot\vec{\sigma}(\theta_\eta(\mu_\eta + d_\eta))^{\vec{L}} \\
& = & \sum_{m \in A} \sum_{|\vec{L}|_\rho = m} P_{(\vec{L})}(a) \cdot\vec{\sigma}(\theta_\eta)^{\vec{L}}\cdot\vec{\sigma}(\mu_\eta + d_\eta)^{\vec{L}} \\
& = & \sum_{m\in A} P_{m, \eta}(\mu_\eta + d_\eta)
\end{eqnarray*}
where $P_{m, \eta}$ is the $\sigma$-polynomial over $K\langle a\rangle$ given by
$$P_{m, \eta}(x) = \sum_{|\vec{L}|_\rho = m} P_{(\vec{L})}(a) \cdot\vec{\sigma}(\theta_\eta)^{\vec{L}}\cdot \vec{\sigma}(x)^{\vec{L}}.$$

Since $P\not\in K$, there is an $m\in A$ such that $P_{m, \eta}\not= 0$. For such $m$, pick $\vec{L} = \vec{L}(m)$ with $|\vec{L}|_\rho = m$ for which $v(P_{(\vec{L})}(a)\cdot\vec{\sigma}(\theta_\eta)^{\vec{L}})$ is minimal, so 
$$P_{m, \eta}(x) = P_{(\vec{L})}(a)\cdot\vec{\sigma}(\theta_\eta)^{\vec{L}}\cdot p_{m, \eta}(\vec{\sigma}(x)),$$
where $p_{m, \eta}(x_0, \ldots, x_n)$ has its coefficients in the valuation ring of $K\langle a\rangle$, with one of its coefficients equal to $1$. Then
$$v(P_{m, \eta}(\mu_\eta + d_\eta)) =  v(P_{(\vec{L})}(a)) + m\cdot\gamma_\eta + v(p_{m, \eta}(\vec{\sigma}(\mu_\eta + d_\eta))).$$

This calculation suggests a new constraint on $\{\mu_\eta\}$, namely that for each $m\in A$ with $P_{m, \eta}\not= 0$, $$v(p_{m, \eta}(\vec{\sigma}(\mu_\eta + d_\eta))) = 0 \;\;\;\; (\mbox{eventually in $\eta$}).$$

Assume this constraint is met. Then Lemma 4.1 yields a fixed $m_0\in A$ such that if $m\in A$ and $m\not= m_0$, then eventually in $\eta$,
$$v(P_{m_0, \eta}(\mu_\eta + d_\eta)) < v(P_{m, \eta}(\mu_\eta + d_\eta))$$

For this $m_0$ we have, eventually in $\eta$,
$$v(P(b_\eta) - P(a)) = v(P_{(\vec{L})}(a)) + m_0\cdot\gamma_\eta, \;\;\;\; \vec{L} = \vec{L}(m_0),$$
which is increasing. So $P(b_\eta)\leadsto P(a)$, as desired.

To have $\{\mu_\eta\}$ satisfy all constraints, we introduce an axiom (scheme) about $\K$ which involves only the residue field $k$ of $\K:$
\begin{center}
\textbf{Axiom 2.} For each integer $d > 0$ there is $y\in k$ such that $\bar{\sigma}^d(y) \not= y$.
\end{center}
By \cite{C}, p. 201, this axiom implies that there are no residual $\bar{\sigma}$-identities at all, that is, for every non-zero $f\in k[x_0, \ldots, x_n]$, there is a $y\in k$ with $f(\bar{\vec{\sigma}}(y))\not= 0$ (and thus the set $\{y\in k : f(\bar{\vec{\sigma}}(y))\not= 0\}$ is infinite). Now note that the $p_m$'s are over $K\langle a\rangle$, and we need $\bar{\mu}_\eta\in k$. The following lemma will take care of this.
\begin{lem}
Let $k\subseteq k'$ be a field extension, and $p(x_0, \ldots, x_n)$ a non-zero polynomial over $k'$. Then there is a non-zero polynomial $f(x_0, \ldots, x_n)$ over $k$ such that whenever $y_0, \ldots, y_n\in k$ and $f(y_0, \ldots, y_n)\not= 0$, then $p(y_0, \ldots, y_n)\not= 0$.
\end{lem}
\begin{proof}
Using a basis $b_1, \ldots, b_m$ of the $k$-vector subspace of $k'$ generated by the coefficients of $p$, we have $p = b_1 f_1 + \cdots + b_m f_m$, with $f_1, \ldots, f_m\in k[x_0, \ldots, x_n].$ Let $f$ be one of the $f_i$'s. Then $f$ has the required property.
\end{proof}

\hspace{-1.2em}Consider an $m\in A$ with non-zero $P_{m, \eta}$, and define
$$q_{m, \eta}(x_0, \ldots, x_n) := p_{m, \eta}(x_0 + d_\eta, \ldots, x_n + \sigma^n(d_\eta)).$$
Then the reduced polynomial
$$\bar{q}_{m, \eta}(x_0, \ldots, x_n) := \bar{p}_m(x_0 + \bar{d}_\eta, \ldots, x_n + \bar{\sigma}^n(\bar{d}_\eta))$$
is also non-zero for each $\eta$. By Lemma 4.2, we can pick a non-zero polynomial $f_\eta(x_0, \ldots, x_n)\in k[x_0, \ldots, x_n]$ such that if $y\in\O_K$ and $f_\eta(\bar{\vec{\sigma}}(\bar{y})) \not= 0$, then $\bar{q}_{m, \eta}(\bar{\vec{\sigma}}(\bar{y}))\not= 0$ for each $m\in A$ with $P_{m, \eta}\not= 0.$
\newline\newline
\textbf{Conclusion:} if for each $\eta$ the element $\mu_\eta\in\O_K$ satisfies $\bar{\mu}_\eta\not= 0, \bar{\mu}_\eta + \bar{d}_\eta \not= 0$, and $f_\eta(\bar{\vec{\sigma}}(\bar{\mu}_\eta)) \not= 0$, then all constraints on $\{\mu_\eta\}$ are met.

Axiom 2 allows us to meet these constraints, even if instead of a single $P(x)$ of order $\le n$ we have finitely many non-constant $\sigma$-polynomials $Q(x)$ of order $\le n$ and we have to meet simultaneously the constraints coming from each of those $Q$'s. This leads to:
\begin{thm}
Suppose $\K$ satisfies Axiom 2. Suppose $\{a_\eta\}$ in $K$ is a pc-sequence and $a_\eta\leadsto a$ in an extension with $\gamma_\eta := v(a - a_\eta)$. Let $\Sigma$ be a finite set of $\sigma$-polynomials $P(x)$ over $K$. 
\begin{itemize}
\item If $\rho$ is transcendental, then $P(a_\eta)\leadsto P(a)$, for all non-constant $P\in K[x]$; more specifically there is a unique $\vec{L}_0 = \vec{L}_0(P)$ such that for all $\vec{I}\not=\vec{L}_0$, eventually
$$v(P(a_\eta) - P(a)) = v(P_{(\vec{L}_0)}(a)) + |\vec{L}_0|_\rho\cdot\gamma_\eta < v(P_{(\vec{I})}(a)) + |\vec{I}|_\rho\cdot\gamma_\eta.$$
\item If $\rho$ is algebraic, then there is a pc-sequence $\{b_\eta\}$ from $K$, equivalent to $\{a_\eta\}$, such that $P(b_\eta)\leadsto P(a)$ for each non-constant $P\in\Sigma$; more specifically there is a unique $m_0 = m_0(P)$ such that for all $\vec{I}$ with $|\vec{I}|_\rho \not= m_0$, eventually
$$v(P(b_\eta) - P(a)) = \min_{|\vec{L}_0|_\rho=m_0} v(P_{(\vec{L}_0)}(a)) + |\vec{L}_0|_\rho\cdot\gamma_\eta < v(P_{(\vec{I})}(a)) + |\vec{I}|_\rho\cdot\gamma_\eta.$$
\end{itemize}
\end{thm}
\begin{cor}
The same result, where $a$ is removed and one only asks that $\{P(b_\eta)\}$ is a pc-sequence. 
\end{cor}
\begin{proof}
By an observation of Macintyre, any pc-sequence in any expansion of valued fields (for example, a valued difference field) has a pseudo-limit in an elementary extension of that expansion. In particular, $\{a_\eta\}$ has a pseudo-limit $a$ in an elementary extension of $\mathcal{K}$. Use this $a$ and Theorem 4.3.
\end{proof}
\hspace{-1.2em}\textbf{Refinement of the Basic Calculation.} The following improvement of the basic calculation will be needed later on.
\begin{thm}
Suppose $\K$ satisfies Axiom 2 and $\rho$ is algebraic. Let $\{a_\eta\}$ be a pc-sequence from $K$ and let $a_\eta\leadsto a$ in some extension. Let $P(x)$ be a $\sigma$-polynomial over $K$ such that
\begin{enumerate}
\item[$(i)$] $P(a_\eta)\leadsto 0$,
\item[$(ii)$] $P_{(\vec{L})}(b_\eta)\not\leadsto 0$, whenever $|\vec{L}|\ge 1$ and $\{b_\eta\}$ is a pc-sequence in $K$ equivalent to $\{a_\eta\}.$
\end{enumerate}
Let $\Sigma$ be a finite set of $\sigma$-polynomials $Q(x)$ over $K$. Then there is a pc-sequence $\{b_\eta\}$ in $K$, equivalent to $\{a_\eta\}$, such that $P(b_\eta)\leadsto 0$, and $Q(b_\eta)\leadsto Q(a)$ for all non-constant $Q$ in $\Sigma$.
\end{thm}
\begin{proof}
By augmenting $\Sigma$, we can assume $P_{(\vec{L})}\in\Sigma$ for all $\vec{L}$. Let $n$ be such that all $Q\in\Sigma$ have order $\le n$. Let $\{\theta_\eta\}$ and $\{d_\eta\}$ be as before. By following the proof in the basic calculation and using Axiom 2, we get non-zero polynomials $f_\eta\in k[x_0, \ldots, x_n]$ and a sequence $\{\mu_\eta\}$ satisfying the constraints
$$\mu_\eta\in\O,\;\;\;\; \bar{\mu}_\eta \not= 0,\;\;\;\; \bar{\mu}_\eta + \bar{d}_\eta \not= 0,\;\;\;\; f_\eta(\bar{\vec{\sigma}}(\bar{\mu}_\eta))\not= 0,$$
such that, by setting $b_\eta := a_\eta + \theta_\eta \mu_\eta$, we have
$$Q(b_\eta)\leadsto Q(a) \;\;\;\;\;\;\;\mbox{for each non-constant $Q\in\Sigma$}.$$
We would like to constrain $\{\mu_\eta\}$ further so that we also have $P(b_\eta)\leadsto 0.$ Letting $A := \{|\vec{L}|_\rho : \vec{L}\in\N^{n+1} \mbox{ and } 1\le |\vec{L}| \le \deg(P)\}$, we have
\begin{eqnarray*}
P(b_\eta) & = & P(a_\eta + \theta_\eta \mu_\eta) \\
& = & P(a_\eta) + \sum_{m\in A}\sum_{|\vec{L}|_\rho = m} P_{(\vec{L})}(a_\eta)\cdot \vec{\sigma}(\theta_\eta \mu_\eta)^{\vec{L}} \\
& = & P(a_\eta) + \sum_{m\in A}\sum_{|\vec{L}|_\rho = m} P_{(\vec{L})}(a_\eta)\cdot \vec{\sigma}(\theta_\eta)^{\vec{L}} \cdot \vec{\sigma}(\mu_\eta)^{\vec{L}} \\
& = & P(a_\eta) + \sum_{m\in A} Q_{m, \eta}(\mu_\eta)
\end{eqnarray*}
where $Q_{m, \eta}$ is the $\sigma$-polynomial over $K$ given by
$$Q_{m, \eta}(x) = \sum_{|\vec{L}|_\rho = m} P_{(\vec{L})}(a_\eta)\cdot \vec{\sigma}(\theta_\eta)^{\vec{L}}\cdot \vec{\sigma}(x)^{\vec{L}}.$$
Since $P_{(\vec{L})}(a_\eta)\leadsto P_{(\vec{L})}(a)$ and $P_{(\vec{L})}(a_\eta)\not\leadsto 0$, $v(P_{(\vec{L})}(a_\eta))$ settles down eventually. Let $\gamma_{\vec{L}}$ be this eventual value. For each $m\in A$ such that $Q_{m, \eta}\not= 0$, let $\vec{L} = \vec{L}(m)$ be such that $P_{(\vec{L})}(a_\eta)\cdot\vec{\sigma}(\theta_\eta)^{\vec{L}}$ has minimal valuation. Then, for such $Q_{m, \eta}$, we can write (eventually in $\eta$),
$$Q_{m, \eta}(x) = c_{m, \eta}\cdot q_{m, \eta}(\vec{\sigma}(x)),$$
where $v(c_{m, \eta}) = \gamma_{\vec{L}} + |\vec{L}|_\rho\cdot\gamma_\eta$ and $q_{m, \eta}$ is a polynomial over $\O$ with at least one coefficient 1. This suggests another constraint on $\{\mu_\eta\}$, namely, for each $m\in A$ such that $Q_{m, \eta}\not= 0$, $v(q_{m, \eta}(\vec{\sigma}(\mu_\eta))) = 0$ (eventually in $\eta$); equivalently, $\bar{q}_{m, \eta}(\bar{\vec{\sigma}}(\bar{\mu}_\eta))\not= 0$. As usual, this constraint can be met by Axiom 2. And then, by Lemma 4.1, we have a unique $m_0$ such that eventually in $\eta$,
$$v\Big(\sum_{m\in A} Q_{m, \eta}(\mu_\eta)\Big) = v(Q_{m_{0}, \eta}(\mu_\eta)) = \gamma_{\vec{L}} + m_0\cdot\gamma_\eta,\;\;\;\;\;\;\;\; \vec{L} = \vec{L}(m_0),$$
which is increasing. Now, if $v(P(a_\eta)) \not= v(Q_{m_{0}, \eta}(\mu_\eta))$, we do nothing. However, if $v(P(a_\eta)) = \gamma_{\vec{L}} + m_0\cdot\gamma_\eta$, then replacing $\mu_\eta$ by a variable $x$, consider
\begin{eqnarray*}
& P(a_\eta) + \sum_{m\in A}\sum_{|\vec{L}|_\rho = m} P_{(\vec{L})}(a_\eta)\cdot\vec{\sigma}(\theta_\eta)^{\vec{L}}\cdot\vec{\sigma}(x)^{\vec{L}} \\
= & P(a_\eta)\Big(1 + \sum_{m\in A}\sum_{|\vec{L}|_\rho = m} P(a_\eta)^{-1} P_{(\vec{L})}(a_\eta)\cdot\vec{\sigma}(\theta_\eta)^{\vec{L}}\cdot\vec{\sigma}(x)^{\vec{L}}\Big) \\
= & P(a_\eta)H_\eta(\vec{\sigma}(x))
\end{eqnarray*}
where $H_\eta(y_0, \ldots, y_n)$ is a polynomial over $\O$ with at least one coefficient 1. So if we add the extra requirement that $\bar{H}_\eta(\bar{\vec{\sigma}}(\bar{\mu}_\eta))\not= 0$, easily fulfilled as before, we get that eventually 
$$v(P(b_\eta)) = \min\{v(P(a_\eta)), v(Q_{m_{0}, \eta}(\mu_\eta))\},$$
and since both of these are increasing, we have $P(b_\eta)\leadsto 0$.
\end{proof}

\bigskip
\section{Around Newton-Hensel Lemma}
For the moment we consider the basic problem of how to start with $a\in K$ and $P(a)\not= 0$, and find $b\in K$ with $v(P(b)) > v(P(a))$.

Before we do that, we need a little notation. Let $\K = (K, \sigma, v)$ be a multiplicative valued difference field. As already mentioned, the automorphism $\sigma$ on $K$ induces an automorphism on the value group $\Gamma$, which we also denote by $\sigma$, as follows:
$$\gamma\mapsto\sigma(\gamma) := v(\sigma(a)), \;\;\;\;\;\;\;\;\mbox{ where } \gamma = v(a) \mbox{ for some } a\in K.$$
Then for any multi-index $\vec{I}=(i_0, i_1, \ldots, i_n)\in\Z^{n+1}$, we have
$$v(\vec{\sigma}(a)^{\vec{I}}) = v(a^{i_0}(\sigma(a))^{i_1}\cdots (\sigma^n(a))^{i_n}) = \sum_{j = 0}^{n} i_j v(\sigma^j(a)) = \sum_{j = 0}^{n} i_j \sigma^j(v(a)) = \sum_{j = 0}^{n} i_j \sigma^j(\gamma).$$
In the multiplicative case, $v(\sigma^{j}(a)) = \rho^j\cdot v(a)$. Thus, we will denote the sum $\sum_{j = 0}^{n} i_j \sigma^j(\gamma)$ by $|\vec{I}|_\rho\cdot \gamma$. Also if $\vec{I} = \vec{I}_0 = (0, \ldots, 0, 1, 0, \ldots 0)$ with $1$ at the $i$-th place, we denote $P_{(\vec{I}_0)}$ by $P_{(i)}$. And then $|\vec{I}_0|_\rho\cdot\gamma = \rho^i\cdot\gamma = \sigma^i(\gamma)$. By abuse of notation, we will often identify $\vec{I}_0$ with $i$. For example, we will write $\vec{J}\not = i$ (for some multi-index $\vec{J}$) to actually mean $\vec{J}\not=\vec{I}_0$. Hopefully, this should be clear from the context.

Let $\K$ be a multiplicative valued difference field. Let $P(x)$ be a $\sigma$-polynomial over $K$ of order $\le n$, and $a\in K$. Let $\vec{I}, \vec{J}, \vec{L}\in\N^{n+1}$.
\begin{defn}
We say $(P, a)$ is in $\sigma$-hensel configuration if $P\not\in K$ and there is $0\le i\le n$ and $\gamma\in\Gamma$ such that
\begin{enumerate}
\item[(i)] $v(P(a)) = v(P_{(i)}(a)) + \rho^i\cdot\gamma \le v(P_{(j)}(a)) + \rho^j\cdot\gamma$ whenever $0\le j\le n$,
\item[(ii)] $v(P_{(\vec{J})}(a)) + |\vec{J}|_\rho\cdot\gamma < v(P_{(\vec{L})}(a)) + |\vec{L}|_\rho\cdot\gamma$ whenever $\vec{0}\not= \vec{J} < \vec{L}$ and $P_{(\vec{J})}\not=0.$
\end{enumerate}
We say $(P, a)$ is in strict $\sigma$-hensel configuration if the inequality in (i) is strict for $j\not=i$.
\end{defn}
\begin{rmk}
Note that if $(P, a)$ is in (strict) $\sigma$-hensel configuration, then $P_{(\vec{J})}(a)\not=0$ whenever $\vec{J}\not=\vec{0}$ and $P_{(\vec{J})}\not=0$, so $P(a)\not= 0$, and therefore $\gamma$ as above satisfies
$$v(P(a)) = \min_{0\le j\le n} v(P_{(j)}(a)) + \rho^j\cdot\gamma,$$
so is unique, and we set $\gamma(P, a) := \gamma$. If $(P, a)$ is not in $\sigma$-hensel configuration, we set $\gamma(P, a) := \infty.$ If $(P, a)$ is in strict $\sigma$-hensel configuration, then $i$ is unique and we set $i(P, a) := i$.
\end{rmk}
\begin{rmk}
Suppose $P$ is non-constant, $P(a)\not=0$, $v(P(a)) > 0$ and $v(P_{(\vec{J})}(a)) = 0$ for all $\vec{J}\not=\vec{0}$ with $P_{(\vec{J})}\not=0$. Then $(P, a)$ is in $\sigma$-hensel configuration with $\gamma(P, a) = v(P(a)) > 0$ and any $i$ with $0\le i\le n$; and for $\rho > 1$, $(P, a)$ is in strict $\sigma$-hensel configuration with $\gamma(P, a) = v(P(a)) > 0$ and $i(P, a) = 0$.
\end{rmk}

Now given $(P, a)$ in (strict) $\sigma$-hensel configuration, we aim to find $b\in K$ such that $v(P(b)) > v(P(a))$ and $(P, b)$ is in (strict) $\sigma$-hensel configuration. This, however, requires an additional assumption on the residue field $k$, namely that $k$ should be linear difference-closed. We will justify later on why this assumption is necessary.
\newline\newline
\textbf{Axiom 3$_n$.} If $\alpha_0, \ldots, \alpha_n\in k$ are not all 0, then the equation
$$1 + \alpha_0 x + \alpha_1\bar{\sigma}(x) + \cdots + \alpha_n \bar{\sigma}^n(x) = 0$$
has a solution in $k$.
\begin{lem}
Suppose $\K$ satisfies Axiom 3$_n$, and $(P, a)$ is in $\sigma$-hensel configuration. Then there is $b\in K$ such that
\begin{enumerate}
\item $v(b - a)\ge\gamma(P, a), \;\;\;\; v(P(b)) > v(P(a))$,
\item either $P(b) = 0$, or $(P, b)$ is in $\sigma$-hensel configuration.
\end{enumerate}
For any such $b$, we have $v(b - a) = \gamma(P, a)$ and $\gamma(P, b) > \gamma(P, a)$.
\end{lem}
\begin{proof}
This is the same proof as \cite{A}, Lemma 4.4. But we include it here for the sake of completeness, and also to set the ground for the next lemma.\newline

\textit{Step 1.} Let $\gamma = \gamma(P, a)$. Pick $\epsilon\in K$ with $v(\epsilon) = \gamma$. Let $b = a + \epsilon u$, where $u\in K$ is to be determined later; we only impose $v(u)\ge 0$ for now. Consider
$$P(b) = P(a) + \sum_{|\vec{J}|\ge 1} P_{(\vec{J})}(a)\cdot \vec{\sigma}(b - a)^{\vec{J}}.$$
Therefore, $P(b) = P(a)\cdot(1 + \sum_{|\vec{J}|\ge 1} c_{\vec{J}}\cdot\vec{\sigma}(u)^{\vec{J}})$, where
$$c_{\vec{J}} = \dfrac{P_{(\vec{J})}(a)\cdot\vec{\sigma}(\epsilon)^{\vec{J}}}{P(a)}.$$
From $v(\epsilon) = \gamma$ and the fact that $(P, a)$ is in $\sigma$-hensel configuration, we obtain $\min_{0\le j\le n} v(c_j) = 0$ and $v(c_{\vec{L}}) > 0$ for $|\vec{L}| > 1$. Then imposing $v(P(b)) > v(P(a))$ forces $\bar{u}$ to be a solution of the equation
$$1 + \sum_{0\le j\le n} \bar{c}_j\cdot\bar{\sigma}^j(x) = 0.$$
By Axiom 3$_n$, we can take $u$ with this property, and then $v(u) = 0$, so $v(b - a) = \gamma(P, a)$, and $v(P(b)) > v(P(a))$.\newline

\textit{Step 2.} Assume that $P(b)\not= 0$. It remains to show that then $(P, b)$ is in $\sigma$-hensel configuration with $\gamma(P, b) > \gamma$. Let $\vec{J}\not=\vec{0}, P_{(\vec{J})}\not= 0$ and consider
$$P_{(\vec{J})}(b) = P_{(\vec{J})}(a) + \sum_{\vec{L}\not=\vec{0}} P_{(\vec{J})(\vec{L})}(a)\cdot\vec{\sigma}(b - a)^{\vec{L}}.$$
Note that $P_{(\vec{J})}(a)\not=0$. Since $\K$ is of equi-characteristic zero, $v(P_{(\vec{J})(\vec{L})}(a)) = v(P_{(\vec{J + L})}(a))$. Therefore, for all $\vec{L}\not=\vec{0}$,
$$v(P_{(\vec{J})(\vec{L})}(a)\cdot\vec{\sigma}(b - a)^{\vec{L}}) > v(P_{(\vec{J})}(a)),$$
hence $v(P_{(\vec{J})}(b)) = v(P_{(\vec{J})}(a))$. Since $P(b)\not=0$, we can pick $\gamma_1\in\Gamma$ such that
$$P(b) = \min_{0\le j\le n} v(P_{(j)}(a)) + \rho^j\cdot\gamma_1.$$
Then $\gamma < \gamma_1 : $ Pick $0\le i\le n$ such that $v(P(a)) = v(P_{(i)}(a)) + \rho^i\cdot\gamma$. So
$$\rho^{i}\cdot\gamma = v(P(a)) - v(P_{(i)}(a)) < v(P(b)) - v(P_{(i)}(a)) \le \rho^{i}\cdot\gamma_1.$$
Also for $\vec{I}, \vec{J}\not=\vec{0}$ and $\theta\in\Gamma$ with $\theta>0$, we have $|\vec{J}|_\rho\cdot\theta < |\vec{L}|_\rho\cdot\theta$ for $\vec{J} < \vec{L}$ (here we are using the fact that $\rho > 0$). Thus the inequality
$$v(P_{(\vec{J})}(a)) + |\vec{J}|_\rho\cdot\gamma < v(P_{(\vec{L})}(a)) + |\vec{L}|_\rho\cdot\gamma$$
together with $\gamma_1 > \gamma$ yields
$$v(P_{(\vec{J})}(a)) + |\vec{J}|_\rho\cdot\gamma_1 < v(P_{(\vec{L})}(a)) + |\vec{L}|_\rho\cdot\gamma_1.$$
Hence, $(P, b)$ is in $\sigma$-hensel configuration with $\gamma(P, b) = \gamma_1$.
\end{proof}
\begin{lem}
Suppose $\K$ satisfies Axiom 3$_n$ and $\rho > 1$, and $(P, a)$ is in $\sigma$-hensel configuration. Then there is $c\in K$ such that
\begin{enumerate}
\item $v(c - a)\ge\gamma(P, a), \;\;\;\; v(P(c)) > v(P(a))$,
\item either $P(c) = 0$, or $(P, c)$ is in strict $\sigma$-hensel configuration.
\end{enumerate}
For any such $c$, we have $v(c - a) = \gamma(P, a)$, $\gamma(P, c) > \gamma(P, a)$; and if $(P, a)$ was already in strict $\sigma$-hensel configuration, then $i(P, c) \le i(P, a)$.
\end{lem}
\begin{proof}
Let $\gamma = \gamma(P, a)$ and $i = i(P, a)$ (in case $(P, a)$ is in strict $\sigma$-hensel configuration). Since $(P, a)$ is in $\sigma$-hensel configuration, by Lemma 5.4, there is $b\in K$ such that $v(b - a) = \gamma(P, a)$, $v(P(b)) > v(P(a))$, $\gamma(P, b) > \gamma(P, a) = \gamma$ and either $P(b) = 0$ or $(P, b)$ is in $\sigma$-hensel configuration.

If $P(b) = 0$, let $c:= b$ and we are done. So suppose $P(b)\not= 0$. Then, letting $\gamma_1 = \gamma(P, b)$, we have for some $0\le j_0\le n$,
$$v(P(b)) = v(P_{(j_0)}(a)) + \rho^{j_0}\cdot\gamma_1 \le v(P_{(j)}(a)) + \rho^j\cdot\gamma_1$$
for all $0\le j\le n$.

If the above inequality is strict for $j\not= j_0$, we are done: Then $(P, b)$ is in strict $\sigma$-hensel configuration with $i(P, b) = j_0$ and $\gamma(P, b) = \gamma_1$. Moreover, if $i < j_0$, then $\rho^i\cdot(\gamma_1 - \gamma) \le \rho^{j_0}\cdot(\gamma_1 - \gamma)$ as $\gamma_1 - \gamma > 0$ and $\rho \ge 1$, and we have
\begin{eqnarray*}
v(P_{(i)}(a)) + \rho^i\cdot\gamma & < & v(P_{(j_0)}(a)) + \rho^{j_0}\cdot\gamma \\
\implies v(P_{(i)}(a)) + \rho^i\cdot\gamma + \rho^i\cdot(\gamma_1 - \gamma) & < & v(P_{(j_0)}(a)) + \rho^{j_0}\cdot\gamma + \rho^{j_0}\cdot(\gamma_1 - \gamma) \\
\implies v(P_{(i)}(a)) + \rho^i\cdot\gamma_1 & < & v(P_{(j_0)}(a)) + \rho^{j_0}\cdot\gamma_1,
\end{eqnarray*}
which is a contradiction. So $j_0 \le i$. Thus, we have, $\gamma(P, b) > \gamma(P, a)$ and $i(P, b)\le i(P, a)$. Let $c := b$.

However, if there is no such unique $j_0$, then it means there are $0\le j_0 < j_1 < \cdots < j_m \le n$ such that
$$v(P(b)) = v(P_{(j_0)}(a)) + \rho^{j_0}\cdot\gamma_1 = v(P_{(j_1)}(a)) + \rho^{j_1}\cdot\gamma_1 = \cdots = v(P_{(j_m)}(a)) + \rho^{j_m}\cdot\gamma_1.$$
Since $(P, b)$ is in $\sigma$-hensel configuration, we can find $b'\in K$ such that $v(P(b')) > v(P(b)) > v(P(a)), \gamma(P, b') > \gamma(P, b) > \gamma(P, a)$, $v(b' - b) = \gamma(P, b)$ and either $P(b') = 0$ or $(P, b')$ is in $\sigma$-hensel configuration. It follows that
\begin{eqnarray*}
v(b' - a) & = & v(b' - b + b - a) \\
& \ge & \min\{v(b' - b), v(b - a)\} \\
& = & \min\{\gamma(P, b), \gamma(P, a)\} \\
\implies v(b' - a) & = & \gamma(P, a) \;\;\;\;\;\;\;\mbox{since, }\; \gamma(P, a) < \gamma(P, b).
\end{eqnarray*}
If $P(b') = 0$, we are done. So suppose $P(b')\not= 0$. Let $\gamma_2 = \gamma(P, b')$. Since $\gamma_2 - \gamma_1 > 0$ and $\rho > 1$ (this is where we crucially use this hypothesis), we have 
$$\rho^{j_0}\cdot(\gamma_2 - \gamma_1) < \rho^{j_1}\cdot(\gamma_2 - \gamma_1) < \cdots < \rho^{j_m}\cdot(\gamma_2 - \gamma_1).$$
But then by doing the same trick as in the previous paragraph, we obtain
$$v(P_{(j_0)}(a)) + \rho^{j_0}\cdot\gamma_2 < v(P_{(j_1)}(a)) + \rho^{j_1}\cdot\gamma_2 < \cdots < v(P_{(j_m)}(a)) + \rho^{j_m}\cdot\gamma_2.$$
Thus, we have succeeded in finding a better approximation $b'$ than $b$ in the sense that $(P, b')$ is in $\sigma$-hensel configuration with its minimal valuation occurring at a possibly lower index than that of $(P, b)$. Since $i(P, a)$ is finite, there are only finitely many possibilities for this index to go down. So by repeating this step finitely many times, we end up at our required $c$ with $v(c - a) = \gamma(P, a)$ such that either $P(c) = 0$ or $(P, c)$ is in strict $\sigma$-hensel configuration with $\gamma(P, c) > \gamma(P, a)$ and $i(P, c)\le i(P, a)$.
\end{proof}
\begin{lem}
Suppose $\K$ satisfies Axiom 3$_n$, and $(P, a)$ is in $\sigma$-hensel configuration. Suppose also there is no $b\in K$ such that $P(b) = 0$ and $v(b - a) = \gamma(P, a)$. Then there is a pc-sequence $\{a_\eta\}$ in $K$ with the following properties:
\begin{enumerate}
\item $a_0 = a$ and $\{a_\eta\}$ has no pseudolimit in $K$;
\item $\{v(P(a_\eta))\}$ is strictly increasing, and thus $P(a_\eta)\leadsto 0$;
\item $v(a_{\eta'} - a_\eta) = \gamma(P, a_\eta)$ whenever $\eta < \eta'$;
\item $(P, a_\eta)$ is in $\sigma$-hensel configuration with $\gamma(P, a_\eta) < \gamma(P, a_{\eta'})$ for $\eta < \eta'$;
\item for any extension $\K'$ of $\K$ and $b, c\in K'$ such that $a_\eta\leadsto b$ and $v(c - b) \ge \gamma(P, b)$, we have $a_\eta\leadsto c$.
\end{enumerate}
\end{lem}
\begin{proof}
We will build the sequence by transfinite recursion. Start with $a_0 := a$. Suppose for some ordinal $\lambda > 0$, we have built the sequence $\{a_\eta\}_{\eta < \lambda}$ such that
\begin{itemize}
\item[(i)] $(P, a_\eta)$ is in $\sigma$-hensel configuration, for all $\eta < \lambda$,
\item[(ii)] $v(a_{\eta'} - a_\eta) = \gamma(P, a_\eta)$ whenever $\eta < \eta' < \lambda$,
\item[(iii)] $v(P(a_{\eta'})) > v(P(a_\eta))$ and $\gamma(P, a_{\eta'}) > \gamma(P, a_\eta)$ whenever $\eta < \eta' < \lambda$.
\end{itemize}
Now we will have to deal with the inductive case. If $\lambda$ is a successor ordinal, say $\lambda = \mu + 1$, then by Lemma 5.4, there is $a_\lambda\in K$ such that $v(a_\lambda - a_\mu) = \gamma(P, a_\mu)$, $v(P(a_\lambda)) > v(P(a_\mu))$ and $\gamma(P, a_\lambda) > \gamma(P, a_\mu)$. Then the extended sequence $\{a_\eta\}_{\eta < \lambda + 1}$ has the above properties with $\lambda + 1$ instead of $\lambda$.

Suppose $\lambda$ is a limit ordinal. Then $\{a_\eta\}$ is a pc-sequence and $P(a_\eta)\leadsto 0$. If $\{a_\eta\}$ has no pseudolimit in $K$, we are done. Otherwise, let $a_\lambda\in K$ be a pseudolimit of $\{a_\eta\}$. Then $v(a_\lambda - a_\eta) = v(a_{\eta + 1} - a_\eta) = \gamma(P, a_\eta)$; also, for any $\eta < \lambda$,
$$P(a_\lambda) = P(a_\eta) + \sum_{|\vec{I}|\ge 1}P_{(\vec{I})}(a_\eta)\cdot \vec{\sigma}(a_\lambda - a_\eta)^{\vec{I}};$$
since $P(a_\eta)$ has the minimal valuation of all the summands, we have $v(P(a_\lambda))\ge v(P(a_\eta))$ for all $\eta < \lambda$. Since $\{v(P(a_\eta))\}_{\eta < \lambda}$ is increasing by inductive hypothesis, we get $v(P(a_\lambda)) > v(P(a_\eta))$ for all $\eta < \lambda$. And then by Step 2 of Lemma 5.4, it follows that $(P, a_\lambda)$ is in $\sigma$-hensel configuration with $\gamma(P, a_\lambda) > \gamma(P, a_\eta)$ for all $\eta < \lambda$. Thus the extended sequence $\{a_\eta\}_{\eta < \lambda + 1}$ satisfies all the above properties with $\lambda + 1$ instead of $\lambda$. Eventually we will have a sequence cofinal in $K$, and hence the building process must come to a stop, yielding a pc-sequence satisfying $(1), (2), (3)$ and $(4)$.

Now $a_\eta\leadsto b$. Thus $v(b - a_\eta) = v(a_{\eta + 1} - a_\eta) = \gamma(P, a_\eta)$ for all $\eta$, and $(P, b)$ is in $\sigma$-hensel configuration with $\gamma(P, b) > \gamma(P, a_\eta)$ for all $\eta$. In particular,
\begin{eqnarray*}
v(c - a_\eta) & = & v(c - b + b - a_\eta)\\
& \ge & \min\{v(c - b), v(b - a_\eta)\}\\
& \ge & \min\{\gamma(P, b), \gamma(P, a_\eta)\} \\
\implies v(c - a_\eta) & = & \gamma(P, a_\eta)
\end{eqnarray*}
Since $\{\gamma(P, a_\eta)\}$ is increasing, we have $a_\eta\leadsto c$.
\end{proof}
It follows similarly (with ideas from the proof of Lemma 5.5) that
\begin{lem}
Suppose $\K$ satisfies Axiom 3$_n$, $\rho > 1$ and $(P, a)$ is in strict $\sigma$-hensel configuration. Suppose also there is no $b\in K$ such that $P(b) = 0$ and $v(b - a) = \gamma(P, a)$. Then there is a pc-sequence $\{a_\eta\}$ in $K$ with the following properties:
\begin{enumerate}
\item $a_0 = a$ and $\{a_\eta\}$ has no pseudolimit in $K$;
\item $\{v(P(a_\eta))\}$ is strictly increasing, and thus $P(a_\eta)\leadsto 0$;
\item $v(a_{\eta'} - a_\eta) = \gamma(P, a_\eta)$ whenever $\eta < \eta'$;
\item $(P, a_\eta)$ is in strict $\sigma$-hensel configuration with $\gamma(P, a_\eta) < \gamma(P, a_{\eta'})$ and $i(P, a_{\eta'})\le i(P, a_\eta) $ for $\eta < \eta'$;
\item for any extension $\K'$ of $\K$ and $b, c\in K'$ such that $a_\eta\leadsto b$ and $v(c - b) \ge \gamma(P, b)$, we have $a_\eta\leadsto c$.
\end{enumerate}
\end{lem}
\begin{defn}
A multiplicative valued difference field $\K$ is called (strict) $\sigma$-henselian if for all $(P, a)$ in (strict) $\sigma$-hensel configuration there is $b\in K$ such that $v(b - a) = \gamma(P, a)$ and $P(b) = 0$.
\end{defn}
By \textbf{Axiom 3} we mean the set $\{\mbox{Axiom } 3_n : n = 0, 1, 2, \ldots\}$. So Axiom 3 is really an axiom scheme and $\K$ satisfies Axiom 3 if and only if $k$ is linear difference closed.
\begin{cor}
If $\K$ is maximally complete as a valued field and satisfies Axiom 3, then $\K$ is $\sigma$-henselian (strict $\sigma$-henselian if $\rho > 1$). In particular, if $\K$ is complete with discrete valuation and satisfies Axiom 3, then $\K$ is $\sigma$-henselian (strict $\sigma$-henselian if $\rho > 1$).
\end{cor}
\begin{lem}
\begin{enumerate}
\item If $\K$ is $\sigma$-henselian, then $\K$ satisfies Axiom 3.
\item If $\K$ satisfies Axiom 3, then $\K$ satisfies Axiom 2.
\end{enumerate}
\end{lem}
\begin{proof}
(1) Assume that $\K$ is $\sigma$-henselian and let $Q(x) = 1 + \alpha_0 x + \alpha_1 \bar{\sigma}(x) + \cdots + \alpha_n \bar{\sigma}^n(x)\in k\langle x\rangle$ such that not all $\alpha_i$'s are zero. We want to find $b\in k$ such that $Q(b) = 0$.

Let $P(a) = 1 + a_0 x + a_1 \sigma(x) + \cdots + a_n \sigma^n(x)$, where for all $i$, $a_i\in K$, $a_i = 0$ if $\alpha_i = 0$, and $v(a_i) = 0$ with $\bar{a}_i = \alpha_i$ if $\alpha_i\not= 0$. It is easy to see that $(P, 0)$ is in $\sigma$-hensel configuration with $\gamma(P, 0) = 0$. By $\sigma$-henselianity, there is $a\in K$ such that $v(a) = 0$ and $P(a) = 0$. Set $b := \bar{a}$.

(2) For $\K$ to satisfy Axiom 2, we need for each $d\in\Z_+$, an element $a\in k$ such that $\bar{\sigma}^d(a) \not= a$. Consider the linear difference polynomial $P_d(x) = \bar{\sigma}^d(x) - x + 1$ over $k$. Since $\K$ satisfies Axiom 3, there is $a\in k$ such that $P_d(a) = 0$, i.e., $\bar{\sigma}^d(a) = a - 1$. In particular, $\bar{\sigma}^d(a)\not= a$.
\end{proof}
\begin{rmk}
\begin{enumerate}
\item If $\Gamma = \{0\}$, then $\K$ is $\sigma$-henselian. 
\item If $\Gamma\not=\{0\}$ and $\K$ is $\sigma$-henselian, then $\K$ satisfies Axiom 3 by Lemma 5.10. In particular, $\bar{\sigma}^n\not= id_k$ for all $n\ge 1$. Thus, $\K$ satisfies Axiom 2 as well.
\item If $\rho > 1$ and $\K$ satisfies Axiom 3, then $\K$ is $\sigma$-henselian iff $\K$ is strict $\sigma$-henselian: the ``only-if'' direction is trivial, and the ``if'' direction follows from Lemma 5.5. However, the ``if'' direction is not true for $\rho = 1$.
\end{enumerate}
\end{rmk}
\begin{defn}
We say $\{a_\eta\}$ is of $\sigma$-algebraic type over $K$ if $P(b_\eta)\leadsto 0$ for some $\sigma$-polynomial $P(x)$ over $K$ and an equivalent pc-sequence $\{b_\eta\}$ in $K$. Otherwise, we say $\{a_\eta\}$ is of $\sigma$-transcendental type.

If $\{a_\eta\}$ is of $\sigma$-algebraic type over $K$, then a minimal $\sigma$-polynomial of $\{a_\eta\}$ over $K$ is a $\sigma$-polynomial $P(x)$ over $K$ with the following properties:
\begin{itemize}
\item[(i)] $P(b_\eta)\leadsto 0$ for some pc-sequence $\{b_\eta\}$ in $K$ equivalent to $\{a_\eta\}$;
\item[(ii)] $Q(b_\eta)\not\leadsto 0$ whenever $Q(x)$ is $\sigma$-polynomial over $K$ of lower complexity than $P(x)$ and $\{b_\eta\}$ is a pc-sequence in $K$ equivalent to $\{a_\eta\}$.
\end{itemize} 
\end{defn}
\begin{lem}
Suppose $\K$ satisfies Axiom 2. Let $\{a_\eta\}$ from $K$ be a pc-sequence of $\sigma$-algebraic type over $K$ with minimal $\sigma$-polynomial $P(x)$ over $K$, and with pseudolimit $a$ in some extension. Let $\Sigma$ be a finite set of $\sigma$-polynomials $Q(x)$ over $K$. Then there is a pc-sequence $\{b_\eta\}$ in $K$, equivalent to $\{a_\eta\}$, such that, with $\gamma_\eta := v(a - a_\eta) : $
\begin{enumerate}
\item[(I)] $v(a - b_\eta) = \gamma_\eta$, eventually, and $P(b_\eta)\leadsto 0;$
\item[(II)] if $Q\in\Sigma$ and $Q\not\in K$, then $Q(b_\eta)\leadsto Q(a);$
\item[(III)] $(P, b_\eta)$ is in $\sigma$-hensel configuration with $\gamma(P, b_\eta) = \gamma_\eta$, eventually;
\item[(IV)] $(P, a)$ is in $\sigma$-hensel configuration with $\gamma(P, a) > \gamma_\eta$ eventually.
\end{enumerate}
If $\rho > 1$, then $(P, b_\eta)$ is actually in strict $\sigma$-hensel configuration. Also there is some $a'$, pseudolimit of $\{a_\eta\}$, such that $(I), (II)$ and $(IV)$ hold with $a$ replaced by $a'$, and $(P, a')$ is in strict $\sigma$-hensel configuration with $\gamma(P, a') > \gamma_\eta$ eventually.
\end{lem}
\begin{proof}
Let $P$ have order $n$. Let us augment $\Sigma$ with all $P_{(\vec{I})}$ for $1\le |\vec{I}|\le\deg(P)$. In the rest of the proof, all multi-indices range over $\N^{n + 1}$. Also since $P$ is a minimal polynomial of $\{a_\eta\}$, there is an equivalent sequence $\{c_\eta\}$ such that $P(c_\eta)\leadsto 0$.

Now if $\rho$ is transcendental, then by Theorem 4.3, $Q(c_\eta)\leadsto Q(a)$ for all $Q\in\Sigma$ and $Q\not\in K$. Let $b_\eta := c_\eta$. Thus, $\{b_\eta\}$ satisfies $(I)$ and $(II)$. And if $\rho$ is algebraic, then by Theorem 4.5, there is a pc-sequence $\{b_\eta\}$, equivalent to $\{c_\eta\}$ (and hence to $\{a_\eta\}$), such that $(I)$ and $(II)$ hold. Theorem 4.3 also shows that in the transcendental case, there is a unique $\vec{L}_0$ such that eventually for all $\vec{I}\not=\vec{L}_0$,
$$v(P(b_\eta) - P(a)) = v(P_{(\vec{L}_0)}(a)) + |\vec{L}_0|_\rho\cdot\gamma_\eta < v(P_{(\vec{I})}(a)) + |\vec{I}|_\rho\cdot\gamma_\eta,$$
and in the algebraic case there is a unique $m_0$ such that eventually for all $\vec{I}$ with $|\vec{I}|_\rho\not = m_0,$
\begin{eqnarray}
& &v(P(b_\eta) - P(a)) = \min_{|\vec{L}_0|_\rho = m_0} v(P_{(\vec{L}_0)}(a)) + m_0\cdot\gamma_\eta < v(P_{(\vec{I})}(a)) + |\vec{I}|_\rho\cdot\gamma_\eta.
\end{eqnarray}
We will show that in either case $|\vec{L}_0| = 1$. Since for $\rho > 1$, there is a unique $\vec{L}_0$ such that $|\vec{L}_0|_\rho = m_0$ and $|\vec{L}_0| = 1$, this gives us that for $\rho > 1$ (both algebraic and transcendental), there is a unique $\vec{L}_0$ such that eventually for all $\vec{I}\not=\vec{L}_0$,
\begin{eqnarray}
&&v(P(b_\eta) - P(a)) = v(P_{(\vec{L}_0)}(a)) + |\vec{L}_0|_\rho\cdot\gamma_\eta < v(P_{(\vec{I})}(a)) + |\vec{I}|_\rho\cdot\gamma_\eta.
\end{eqnarray}
This actually gives the strict $\sigma$-hensel configuration of $(P, b_\eta)$ for $\rho > 1$.

For any $\vec{I}$ such that $P_{(\vec{I})}\not= 0$, we claim that if $\vec{I} < \vec{J}$, then 
$$v(P_{(\vec{I})}(a)) + |\vec{I}|_\rho\cdot\gamma_\eta < v(P_{(\vec{J})}(a)) + |\vec{J}|_\rho\cdot\gamma_\eta$$
eventually: Theorem 4.3 with $\Sigma = \{P, P_{(\vec{I})}\}$ shows that we can arrange that our sequence $\{b_\eta\}$ also satisfies
$$v(P_{(\vec{I})}(b_\eta) - P_{(\vec{I})}(a)) \le v(P_{(\vec{I})(\vec{L})}(a)) + |\vec{L}|_\rho\cdot \gamma_\eta,$$
eventually for all $\vec{L}$ with $|\vec{L}|\ge 1$. Since $v(P_{(\vec{I})}(b_\eta)) = v(P_{(\vec{I})}(a))$ eventually (as $P$ is a minimal polynomial for $\{a_\eta\}$), this yields
$$v(P_{(\vec{I})}(a)) \le v(P_{(\vec{I})(\vec{L})}(a)) + |\vec{L}|_\rho\cdot\gamma_\eta =  v(P_{(\vec{I} + \vec{L})}(a)) + |\vec{L}|_\rho\cdot\gamma_\eta.$$
For $\vec{L}$ with $\vec{I} + \vec{L} = \vec{J}$, this yields
$$v(P_{(\vec{I})}(a)) \le v(P_{(\vec{J})}(a)) + |\vec{J} - \vec{I}|_\rho\cdot\gamma_\eta.$$
As $\{\gamma_\eta\}$ is increasing, we have eventually in $\eta$,
$$v(P_{(\vec{I})}(a)) < v(P_{(\vec{J})}(a)) + |\vec{J} - \vec{I}|_\rho\cdot\gamma_\eta.$$
Since eventually $v(P_{(\vec{I})}(b_\eta)) = v(P_{(\vec{I})}(a))$, we have
\begin{eqnarray*}
v(P_{(\vec{I})}(b_\eta)) + |\vec{I}|_\rho\cdot\gamma_\eta & < & v(P_{(\vec{J})}(b_\eta)) + |\vec{J}|_\rho\cdot\gamma_\eta, \;\;\;\mbox{and}\\
v(P_{(\vec{I})}(a)) + |\vec{I}|_\rho\cdot\gamma_\eta & < & v(P_{(\vec{J})}(a)) + |\vec{J}|_\rho\cdot\gamma_\eta
\end{eqnarray*}
It follows that $|\vec{L_0}|= 1$ (for $\rho = 1$, this means $m_0 = 1$). In particular, we have established (1) with $m_0 = 1$ for $\rho = 1$, and (2) for $\rho > 1$. Since $P(b_\eta)\leadsto 0$, this yields $v(P(a)) > v(P(b_\eta))$ eventually, i.e, $v(P(b_\eta) - P(a)) = v(P(b_\eta))$. It follows from this and $(1)$ that $(P, b_\eta)$ is in $\sigma$-hensel configuration eventually with $\gamma(P, b_\eta) = \gamma_\eta$; and it follows from $(2)$ that for $\rho > 1$, $(P, b_\eta)$ is in strict $\sigma$-hensel configuration.

Finally by Step 2 of Lemma 5.4, it follows that $(P, a)$ is also in $\sigma$-hensel configuration with $\gamma(P, a) > \gamma_\eta$ eventually; and for $\rho > 1$, if $(P, a)$ is already in strict $\sigma$-hensel configuration, we are done. Otherwise follow the proof of Lemma 5.5 to find the required $a'$.
\end{proof}

\bigskip
\section{Immediate Extensions}
Throughout this section, $\K = (K, \Gamma, k; v, \pi)$ is a  multiplicative valued difference field satisfying Axiom 2. Note that then any immediate extension of $\K$ also satisfies Axiom 2. We state here a few basic facts on immediate extensions.
\begin{lem}
Let $\{a_\eta\}$ from $K$ be a pc-sequence of $\sigma$-transcendental type over $K$. Then $\K$ has a proper immediate extension $(K\langle a\rangle, \Gamma, k; v_a, \pi_a)$ such that:
\begin{enumerate}
\item $a$ is $\sigma$-transcendental over $K$ and $a_\eta\leadsto a$;
\item for any extension $(K_1, \Gamma_1, k_1; v_1, \pi_1)$ of $\K$ and any $b\in K_1$ with $a_\eta\leadsto b$, there is a unique embedding
$$(K\langle a\rangle, \Gamma, k; v_a, \pi_a)\longrightarrow (K_1, \Gamma_1, k_1; v_1, \pi_1)$$
over $K$ that sends $a$ to $b$.
\end{enumerate} 
\end{lem}
\begin{proof} See \cite{AD}, Lemma 6.2. (All that is needed in the proof is the pseudo-continuity of the $\sigma$-polynomials (upto equivalent sequences). So the same proof works here.)
\end{proof}
As a consequence of both $(1)$ and $(2)$ of Lemma 6.1, we have:
\begin{cor}
Let $a$ from some extension of $\K$ be $\sigma$-algebraic over $K$ and let $\{a_\eta\}$ be a pc-sequence in $K$ such that $a_\eta\leadsto a$. Then $\{a_\eta\}$ is of $\sigma$-algebraic type over $K$.
\end{cor}
\begin{lem}
Let $\{a_\eta\}$ from $K$ be a pc-sequence of $\sigma$-algebraic type over $K$, with no pseudolimit in $K$. Let $P(x)$ be a minimal $\sigma$-polynomial of $\{a_\eta\}$ over $K$. Then $\K$ has a proper immediate extension $(K\langle a\rangle, \Gamma, k; v_a, \pi_a)$ such that:
\begin{enumerate}
\item $P(a) = 0$ and $a_\eta\leadsto a$;
\item for any extension $(K_1, \Gamma_1, k_1; v_1, \pi_1)$ of $\K$ and any $b\in K_1$ with $P(b) = 0$  and $a_\eta\leadsto b$, there is a unique embedding
$$(K\langle a\rangle, \Gamma, k; v_a, \pi_a)\longrightarrow (K_1, \Gamma_1, k_1; v_1, \pi_1)$$
over $K$ that sends $a$ to $b$.
\end{enumerate} 
\end{lem}
\begin{proof} See \cite{AD}, Lemma 6.4.
\end{proof}
\begin{defn}
$\K$ is said to be $\sigma$-algebraically maximal if it has no proper immediate $\sigma$-algebraic extension; and $\K$ is said to be maximal if it has no proper immediate extension.
\end{defn}
\begin{cor}
\begin{enumerate}
\item $\K$ is $\sigma$-algebraically maximal if and only if each pc-sequence in $K$ of $\sigma$-algebraic type over $K$ has a pseudolimit in $K$;
\item if $\K$ satisfies Axiom 3 and is $\sigma$-algebraically maximal, then $\K$ is $\sigma$-henselian.
\end{enumerate}
\end{cor}
\begin{proof}
(1) The ``only if'' direction follows from Lemma 6.3. For the ``if'' direction, suppose for a contradiction that $\K_1 := (K_1, \Gamma, k; v_1, \pi_1)$ is a proper immediate $\sigma$-algebraic extension of $\K$. Since the extension is proper, there is $a\in K_1\setminus K$. Since the extension is immediate, we can find a pc-sequence $\{a_\eta\}$ from $K$ such that $a_\eta\leadsto a$. Since the extension is $\sigma$-algebraic, $a$ is $\sigma$-algebraic over $K$. Then by Corollary 6.2, $\{a_\eta\}$ is of $\sigma$-algebraic type over $K$. So by assumption, there is $b\in K$ such that $a_\eta\leadsto b$. But then by part (2) of Lemma 6.3, we have
$$(K\langle a\rangle, \Gamma, k; v_a, \pi_a) \cong (K\langle b\rangle, \Gamma, k; v_b, \pi_b) \cong (K, \Gamma, k; v, \pi),$$
i.e., $a\in K$, a contradiction.

(2) Let $P(x)$ be a $\sigma$-polynomial over $K$ of order $\le n$, and $a\in K$ be such that $(P, a)$ is in $\sigma$-hensel configuration. If there is no $b\in K$ such that $v(b - a) = \gamma(P, a)$ and $P(b) = 0$, then by Lemma 5.6, there is a $\sigma$-algebraic pc-sequence $\{a_\eta\}$ in $K$ such that $\{a_\eta\}$ has no pseudolimit in $K$. But then by part (1) of this corollary, $K$ is not $\sigma$-algebraically maximal, a contradiction.
\end{proof}

It is clear that $\K$ has $\sigma$-algebraically maximal immediate $\sigma$-algebraic extensions, and also maximal immediate extensions. Provided that $\K$ satisfies Axiom 3, both kinds of extensions are unique up to isomorphism, but for this we need one more lemma:
\begin{lem}
Let $\K'$ be a $\sigma$-algebraically maximal extension of $\K$ satisfying Axiom 3. Let $\{a_\eta\}$ from $K$ be a pc-sequence of $\sigma$-algebraic type over $K$, with no pseudolimit in $K$, and with minimal $\sigma$-polynomial $P(x)$ over $K$. Then there exists  $b\in K'$ such that $a_\eta\leadsto b$ and $P(b) = 0$.
\end{lem}
\begin{proof}
By Corollary 6.5 (1), there exist $a\in K'$ such that $a_\eta\leadsto a$. If $P(a) = 0$, we are done. So let's assume $P(a)\not= 0$. But then by Lemma 5.13(IV), $(P, a)$ is in $\sigma$-hensel configuration with $\gamma(P, a) > v(a - a_\eta)$ eventually. Since $\K'$ satisfies Axiom 3, by Corollary 6.5 (2) there is $b\in K'$ such that $v(b - a) = \gamma(P, a)$ and $P(b) = 0$. Finally $v(b - a_\eta) = v(b - a + a - a_\eta) = v(a - a_\eta)$, since $v(b - a) = \gamma(P, a) > v(a - a_\eta)$. Thus, $a_\eta\leadsto b$.
\end{proof}

Together with Lemmas 6.1 and 6.3, this yields:
\begin{thm}
\begin{enumerate}
\item Suppose $\K'$ is a proper immediate $\sigma$-henselian extension of $\K$, and let $a\in K'\setminus K$. Let $\K_1$ be a $\sigma$-henselian extension of $\K$ satisfying Axiom 2, such that every pc-sequence from $\K_1$ of length at most card$(\Gamma)$ has a pseudolimit in $\K_1$. Then $\K\langle a\rangle$ embeds in $\K_1$ over $\K$.
\item Suppose $\K'$ is a proper immediate $\sigma$-henselian $\sigma$-algebraic extension of $\K$, and let $a\in K'\setminus K$. Let $\K_1$ be a $\sigma$-henselian extension of $\K$ satisfying Axiom 2, such that every pc-sequence of $\sigma$-algebraic type over $\K_1$ and of length at most card$(\Gamma)$ has a pseudolimit in $\K_1$. Then $\K\langle a\rangle$ embeds in $\K_1$ over $\K$.
\end{enumerate}
\end{thm}
\begin{proof}
(1) By the classical theory, there is a pc-sequence $\{a_\eta\}$ from $K$ such that $a_\eta\leadsto a$ and $\{a_\eta\}$ has no pseudolimit in $K$. By assumption, there is $b\in K_1$ such that $a_\eta\leadsto b$.

If $\{a_\eta\}$ is of $\sigma$-transcendental type, then by Corollary 6.2, both $a$ and $b$ must be $\sigma$-transcendental over $K$. Now apply Lemma 6.1.

If $\{a_\eta\}$ is of $\sigma$-algebraic type, let $P(x)$ be a minimal polynomial for $\{a_\eta\}$. By Theorem 4.5, we get an equivalent pc-sequence $\{b_\eta\}$ from $K$ with $b_\eta\leadsto a$, such that $P(b_\eta)\leadsto 0$, $P(b_\eta)\leadsto P(a)$, $P_{(\vec{L})}(b_\eta)\leadsto P_{(\vec{L})}(a)$ (but not to 0) for all $|\vec{L}|\ge 1$, $(P, b_\eta)$ is in $\sigma$-hensel configuration eventually, and either $P(a) = 0$, or $(P, a)$ is also in $\sigma$-hensel configuration with $\gamma(P, a) > \gamma(P, b_\eta)$ eventually. If $P(a) = 0$, we are done. Otherwise, since $\K'$ is $\sigma$-henselian, we have $a'\in K'$ such that $P(a') = 0$ and $v(a' - a) = \gamma(P, a)$. Since $\gamma(P, a) > \gamma(P, b_\eta)$ eventually, we have $b_\eta\leadsto a'$. Thus, in either case, we have $a'\in K'$ such that $P(a') = 0$ and $b_\eta\leadsto a'$. Similarly, we get $b'\in K_1$ such that $b_\eta\leadsto b'$ and $P(b') = 0$.

By Lemma 6.3, $K\langle a'\rangle$ is isomorphic to $K\langle b'\rangle$ as multiplicative valued fields over $K$, with $a'$ mapped to $b'$.

Now, $a$ is immediate over $K\langle a'\rangle$. If it is not already in $K\langle a'\rangle$, then we may repeat the argument and conclude by a standard maximality argument.\newline

(2) By Corollary 6.2, there is a pc-sequence $\{a_\eta\}$ from $K$ of $\sigma$-algebraic type pseudoconverging to $a$, but with no pseudolimit in $K$. Then the calculation in (1) works noting that every extension or pc-sequence considered will be of $\sigma$-algebraic type.
\end{proof}
\begin{cor}
Suppose $\K$ satisfies Axiom 3. Then all its maximal immediate extensions are isomorphic over $\K$, and all its $\sigma$-algebraically maximal immediate $\sigma$-algebraic extensions are isomorphic over $\K$.
\end{cor}
\begin{proof}
We have already noticed the existence of both kinds of maximal immediate extensions. By Corollary 6.5(2), they are also $\sigma$-henselian. The desired uniqueness then follows by a standard maximality argument using Theorem 6.7 (1) and (2).
\end{proof}

We now state minor variants of the last two results using the notion of saturation from model theory.
\begin{lem}
Let $\K'$ be a $|\Gamma|^+$-saturated $\sigma$-henselian extension of $\K$. Let $\{a_\eta\}$ from $K$ be a pc-sequence of $\sigma$-algebraic type over $K$, with no pseudolimit in $K$, and with minimal $\sigma$-polynomial $P(x)$ over $K$. Then there exists $b\in K'$ such that $a_\eta\leadsto b$ and $P(b) = 0$. 
\end{lem}
\begin{proof}
By the saturation assumption, there is a pseudolimit  $a\in \K'$ of $\{a_\eta\}$. If $P(a) = 0$, we are done. So let's assume $P(a)\not= 0$. But then by Lemma 5.13(IV), $(P, a)$ is in $\sigma$-hensel configuration with $\gamma(P, a) > v(a - a_\eta)$ eventually. Since $\K'$ is $\sigma$-henselian, there is $b\in K'$ such that $v(b - a) = \gamma(P, a)$ and $P(b) = 0$. Finally, $a_\eta\leadsto b$, since $v(b - a) = \gamma(P, a) > v(a - a_\eta)$.
\end{proof}

In combination with Lemmas 6.1 and 6.3, this yields:
\begin{cor}
Suppose that $\K$ satisfies Axiom 3 and $\K'$ is a $|\Gamma|^+$-saturated $\sigma$-henselian extension of $\K$. Let $\K^*$ be a maximal immediate extension of $\K$. Then $\K^*$ can be embedded in $\K'$ over $\K$.
\end{cor}

\bigskip
\section{Example and Counter-example}
We will now show the consistency of our axioms by building models for our theory. The canonical models we have in mind are the generalized power series fields $k((t^{\Gamma}))$, also known as the Hahn series. Given any difference field $k$ of characteristic zero with automorphism $\bar{\sigma}$, and any $MODAG$ $\Gamma$ with automorphism $\sigma(\gamma) = \rho\cdot\gamma$, we form the multiplicative difference valued field $k((t^\Gamma))$ as follows.

As a set, $k((t^\Gamma)) := \{f:\Gamma\to k\; | \mbox{ supp}(f) := \{x\in\Gamma : f(x)\not= 0\}$ is well-ordered in the ordering induced by $\Gamma\}.$

An element $f\in k((t^\Gamma))$ is thought of as a formal power series
\begin{eqnarray*}
f & \leftrightarrow & \sum_{\gamma\in\Gamma} f(\gamma)t^\gamma \\
(f + h)(\gamma) & := & f(\gamma) + h(\gamma) \\
(fh)(\gamma) & := & \sum_{\alpha + \beta = \gamma} f(\alpha)h(\beta) \\
v(f) & := & \min\mbox{ supp}(f) \\
\sigma(f) & := & \sum_{\gamma\in\Gamma} \bar{\sigma}(f(\gamma))t^{\rho\cdot\gamma}
\end{eqnarray*}

If we choose $\rho\ge 1$, $k((t^\Gamma))$ satisfies Axiom 1. Also if we impose that $\bar{\sigma}$ is a linear difference closed automorphism on $k$, then $k((t^\Gamma))$ satisfies Axiom 2 and Axiom 3 as well. Moreover, using the fact that $k((t^\Gamma))$ is maximally complete \cite{N}, it follows from Corollary 5.9 that $k((t^\Gamma))$ is $\sigma$-henselian for $\rho\ge 1$, and strict $\sigma$-henselian for $\rho > 1$. Also note that the residue field of $k((t^\Gamma))$ is $k$, and the value group is $\Gamma$.
\newline\newline

Now we will justify why we need Axiom 3 (at least for the case $\rho > 1$). We will provide an example that shows why Axiom 3 cannot be dropped. This example is adapted from \cite{A}, Example 5.11.
\newline

\begin{example}
Let $\rho$ be any element of a real-closed field and $\rho > 1$. Let $\Gamma$ be the $MODAG$ generated by $\rho$ over $\Z$. Thus we construe $\Gamma$ as the ordered difference group $\Z[\rho, \rho^{-1}]$ with the order induced by the cut of $\rho$. Let $k$ be any field of characteristic zero, construed as a difference field equipped with its identity automorphism. And let $\K$ be the multiplicative valued difference field $(k((t^\Gamma)), \Gamma, k; v, \pi)$.

For each $n$, let $\Gamma_n := \rho^{-n}\Z[\rho]$ and let $\K_n$ be the multiplicative valued difference field $k((t^{\Gamma_n}), \Gamma_n, k; v, \pi)$. Let
$$\K_\infty := \Big(\bigcup_n k((t^{\Gamma_n})), \Gamma, k; v, \pi\Big).$$
Then $\K_\infty$ equipped with the restriction of $\sigma$, is a valued difference subfield of $\K$ and $\sigma$ is multiplicative. Let us define a sequence $\{a_n\}$ as follows:
$$a_n = \sum_{i = 1}^{n} t^{-\rho^{-i}}.$$
We claim that $\{a_n\}$ is a pc-sequence : Note that since $\rho > 1$, we have for $i < j\in\N$, $v(t^{-\rho^{-i}}) = -\rho^{-i} < -\rho^{-j} = v(t^{-\rho^{-j}})$. Hence, $v(a_{n + 1} - a_n) = v(t^{-\rho^{-(n+1)}}) = -\rho^{-(n + 1)}$, which is increasing as $n\to\infty$.

Also it is clear that $\{a_n\}$ has no pseudolimit in $\K_\infty$. Moreover, since $\sigma(t^{-\rho^{-i}}) = t^{-\rho^{-i + 1}}$, we have for
$$P(x) = \sigma(x) - x - t^{-1},$$
$P(a_n) = t^{-\rho^{-n}}\leadsto 0$, and hence $\{a_n\}$ is of $\sigma$-algebraic type over $\K_\infty$.
Now $\K_\infty$ is a union of henselian valued fields, and hence is henselian. Moreover it is of characteristic zero. Hence $\K_\infty$ is algebraically maximal. Therefore, $P(x)$ is a minimal $\sigma$-polynomial of $\{a_n\}$ over $\K_\infty$. Also,
$$P(a_n) + 1\leadsto 0,$$
and so $P(x) + 1$ is a minimal $\sigma$-polynomial of $\{a_n\}$ over $\K_\infty$ as well. By Lemma 6.3, there are immediate extensions $\K_\infty\langle a\rangle, \K_\infty\langle a'\rangle$ of $\K_\infty$ such that $a_n\leadsto a$, $P(a) = 0$, and $a_n\leadsto a'$, $P(a') + 1= 0$. Let $\L_1, \L_2$ be $\sigma$-algebraically maximal, immediate, $\sigma$-algebraic extensions of $\K_\infty\langle a\rangle, \K_\infty\langle a'\rangle$ respectively.

Now we claim that $\L_1$ and $\L_2$ are not isomorphic over $\K_\infty$. Suppose for a contradiction that they are isomorphic. Then there is $b\in\L_1$ such that $P(b) + 1 = 0$. Since $P(a) = 0$ we have
$$Q(a, b) := \sigma(b - a) - (b - a) + 1 = 0.$$
We claim that this is only possible when $v(b - a) = 0$ : if $v(b - a) > 0$, then since $\rho > 1$, we have $v(\sigma(b - a)) > v(b - a) > 0 = v(1)$. Hence, $v(Q(a, b)) = v(1) = 0$, and thus $Q(a, b)\not= 0$, a contradiction; similarly, if $v(b - a) < 0$, then $v(\sigma(b - a)) < v(b - a) < 0 = v(1)$, in which case again $v(Q(a, b)) = 0$, a contradiction. Thus, $v(b - a) = 0$.

But then, $\overline{b - a}\in k$ and $\overline{b - a}$ is a solution of
$$\bar{\sigma}(x) - x + 1 = 0,$$
which is impossible since $\bar{\sigma} = id$, contradiction.

Here we considered a particular instance of failure of Axiom 3; namely, when $\bar{\sigma}$ is the identity, the above $\bar{\sigma}$-linear equation does not have a solution in $k$. However, one can produce a similar construction for any non-degenerate inhomogeneous $\bar{\sigma}$-linear equation which does not have a solution in $k$.
\end{example}

\bigskip
\section{Extending Residue Field and Value Group}
Let $\K = (K, \Gamma, k; v, \pi)$ and $\K' = (K', \Gamma', k'; v', \pi')$ be two multiplicative valued difference fields (with the same $\rho$); let $\O$ and $\O'$ be their respective ring of integers, and let $\sigma$ denote both their difference operators. Let $\E = (E, \Gamma_E, k_E; v, \pi)$ be a common multiplicative valued difference subfield of both $\K$ and $\K'$, that is, $\E\le\K, \E\le\K'$. Then we have:
\begin{lem}
Let $a\in\O$ and assume $\alpha = \pi(a)$ is $\bar{\sigma}$-transcendental over $k_E$. Then
\begin{itemize}
\item $v(P(a)) = \min_{\vec{L}}\{v(b_{\vec{L}})\}$ for each $\sigma$-polynomial $P(x) = \sum b_{\vec{L}}\vec{\sigma}(x)^{\vec{L}}$ over $E$;
\item $v(E\langle a\rangle^\times) = v(E^\times) = \Gamma_E$, and $\E\langle a\rangle$ has residue field $k_E\langle\alpha\rangle$;
\item if $b\in\O'$ is such that $\beta = \pi(b)$ is $\bar{\sigma}$-transcendental over $k_E$, then there is a valued difference field isomorphism $\E\langle a\rangle\to\E'\langle b\rangle$ over $\E$ sending $a$ to $b$.
\end{itemize}
\end{lem}
\begin{proof}
See \cite{AD}, Lemma 2.5.
\end{proof}
\begin{lem}
Let $P(x)$ be a non-constant $\sigma$-polynomial over the valuation ring of $E$ whose reduction $\bar{P}(x)$ has the same complexity as $P(x)$. Let $a\in\O, b\in\O'$, and assume that $P(a) = 0, P(b) = 0$, and that $\bar{P}(x)$ is a minimal $\bar{\sigma}$-polynomial of $\alpha := \pi(a)$ and of $\beta := \pi'(b)$ over $k_E$. Then
\begin{itemize}
\item $\E\langle a\rangle$ has value group $v(E^\times) = \Gamma_E$ and residue field $k_E\langle\alpha\rangle$;
\item if there is a difference field isomorphism $k_E\langle\alpha\rangle\to k_E\langle\beta\rangle$ over $k_E$ sending $\alpha$ to $\beta$, then there is a valued difference field isomorphism $\E\langle a\rangle\to\E\langle b\rangle$ over $\E$ sending $a$ to $b$.
\end{itemize}
\end{lem}
\begin{proof}
See \cite{AD}, Lemma 2.6.
\end{proof}

Now we will show how to extend the value group. Recall that $\Gamma$ is a model of $MODAG$. Before stating the results, we need a couple of definitions.
\begin{defn}
For a given $\sigma$-polynomial $P(x) = \sum b_{\vec{L}}\vec{\sigma}(x)^{\vec{L}}$ over $K$ and $a\in K$, we say $a$ is generic for $P$ if $v(P(a)) = \min\{v(b_{\vec{L}}) + |\vec{L}|_\rho\cdot v(a)\}$.
\end{defn}
\begin{defn}
An element $a\in K$ (or $K'$) is said to be generic over $\E$ if $a$ is generic for all $\sigma$-polynomials $P(x) = \sum b_{\vec{L}}\vec{\sigma}(x)^{\vec{L}}$ over $E$.
\end{defn}
\begin{lem}
Assume $\K$ satisfies Axiom 2. Then, for each $\gamma\in\Gamma$ and $P(x) = \sum b_{\vec{L}}\vec{\sigma}(x)^{\vec{L}}$ over $K$, there is $a\in K$ such that $v(a) = \gamma$ and $a$ is generic for $P$.
\end{lem}
\begin{proof}
Let $c\in K$ be such that $v(c) = \gamma$. If $c$ is already generic for $P$, set $a := c$ and we are done. Otherwise, pick $\epsilon\in K$ such that $v(\epsilon) = 0$ (we will decide later how to choose $\epsilon$) and set $a := c\epsilon$. Note that $v(a) = v(c) = \gamma$. Then, $P(a) = \sum b_{\vec{L}}\vec{\sigma}(c)^{\vec{L}}\vec{\sigma}(\epsilon)^{\vec{L}}$. Choosing $d\in K^\times$ such that $v(d) = \min\{v(b_{\vec{L}}\vec{\sigma}(c)^{\vec{L}})\} = \min\{v(b_{\vec{L}}) + |\vec{L}|_\rho\cdot\gamma\}$, we can write $P(a) = dQ_P(\epsilon)$, where $Q_P(\epsilon)$ is over the ring of integers of $K$. Consider $\overline{Q_P}(\bar{\epsilon})$, a $\bar{\sigma}$-polynomial over $k$; choose $\bar{\epsilon}\in k$ such that $\overline{Q_P}(\bar{\epsilon})\not= 0$, which is possible since $\K$ satisfies Axiom 2. Let $\epsilon\in K$ be such that $\pi(\epsilon) = \bar{\epsilon}$. Then $v(Q_P(\epsilon)) = 0$, and hence $v(P(a)) = v(d) = \min\{v(b_{\vec{L}}) + |\vec{L}|_\rho\cdot\gamma\}$. Thus, $a$ is generic for $P$ and $v(a) = \gamma$. 
\end{proof}
\begin{rmk}
It is clear from the proof of Lemma 8.5 that if we replace $P(x)$ by a finite set of $\sigma$-polynomials $\{P_1(x), \ldots, P_m(x)\}$ of possibly different orders, then by choosing $\bar{\epsilon}\in k$ such that it doesn't solve any of the related $m$ equations $\overline{Q_{P_i}}(x) = 0$ over $k$ (which is again possible to do as $\K$ satisfies Axiom 2), we can find $a\in K$ such that $a$ is generic for $\{P_1, \ldots, P_m\}$. 
\end{rmk}

\begin{defn}
Let $P(x) = \sum b_{\vec{L}}\vec{\sigma}(x)^{\vec{L}}$ be a $\sigma$-polynomial over $K$ and $a\in K$. Write $P(a x) = dQ_P(x)$, where $d\in K$ is such that $v(d) = \min\{v(b_{\vec{L}}) + |\vec{L}|_\rho\cdot v(a)\}$. Then $Q_P(x)$ is a $\sigma$-polynomial over $\O_K$, and thus $\overline{Q_P}(x)$ is a $\bar{\sigma}$-polynomial over $k$. We say $\overline{Q_P}(x)$ is a $k$-$\bar{\sigma}$-polynomial corresponding to $(P, a)$.
\end{defn}

\begin{lem}
Let $\gamma\in\Gamma\setminus\Gamma_E$. Let $\kappa$ be an infinite cardinal such that $|k_E|\le\kappa$. Assume $\K, \K'$ satisfy Axiom 2 and are $\kappa^+$-saturated. Then 
\begin{enumerate}
\item[(i)] there is $a\in K$, generic over $\E$, with $v(a) = \gamma$;
\item[(ii)] $E\langle a\rangle$ has value group $\Gamma_E\langle\gamma\rangle$, and residue field $k_{E\langle a\rangle}$ of size $\le\kappa$;
\item[(iii)] if $\gamma'\in\Gamma'$ is such that there is a valued difference group isomorphism $\Gamma_E\langle\gamma\rangle\to\Gamma_E\langle\gamma'\rangle$ over $\Gamma_E$ (in the language of $MODAG$), and $a'\in K'$ is such that $a'$ is generic over $\E$ with $v(a') = \gamma'$, then there is a valued difference field isomorphism $\E\langle a\rangle\to\E\langle a'\rangle$ over $\E$ sending $a$ to $a'$.
\end{enumerate}
\end{lem}
\begin{proof}
(i) Fix $c\in K$ such that $v(c) = \gamma$. For each $\sigma$-polynomial $P(x)$ over $E$, let $\overline{Q_P}(x)$ be a $k$-$\bar{\sigma}$-polynomial corresponding to $(P, c)$,  and define
$$\varphi_P(y) := \overline{Q_P}(y)\not=0$$
i.e. $\varphi_P(y)$ is the first-order formula with parameters from $k$ saying ``$y$ is not a root of $\overline{Q_P}$''. Let
\begin{eqnarray*}
p(y) & := & \{\varphi_P(y)\; |\; P \mbox{ is a } \sigma\mbox{-polynomial over } E\}.
\end{eqnarray*}
By Axiom 2, $p(y)$ is finitely consistent, and hence consistent. So it's a type over $E$. Moreover by cardinality considerations, $|p(y)|\le\kappa^{<\omega} = \kappa$ (since $\kappa$ is infinite). Since $\K$ is $\kappa^+$-saturated, $p(y)$ is realized in $\K$. In particular, there is $y\in k$ such that $y$ is not a root of any $\overline{Q_P}$. Choosing $\epsilon\in\O$ with $\pi(\epsilon) = y$ and setting $a := c\epsilon$, we then have that $v(a) = \gamma$ and $a$ is generic for all $\sigma$-polynomials $P(x)$ over $E$, i.e., $a$ is generic over $\E$.

(ii) Since $a$ is generic over $\E$, it is clear that $v(E\langle a\rangle^\times) = \Gamma_E\langle\gamma\rangle$, which clearly has size at most $\kappa$. Moreover, since each element of the residue field comes from an element of valuation zero, the size of $k_{E\langle a\rangle}$ is at most the size of the set $\{P(x)\;|\; P \mbox{ is a } \sigma\mbox{-polynomial over } E \mbox{ and } v(P(a)) = 0\}$, which, again by cardinality considerations, is at the most $\kappa$. Thus $|k_{E\langle a\rangle}|\le\kappa$.

(iii) Finally notice that if $b$ is generic over $\E$, then $v(P(b))\not=\infty$ for any $\sigma$-polynomial $P(x)$ over $E$; i.e., $P(b)\not= 0$. So $b$ is $\sigma$-transcendental over $E$. In particular, $a$ and $a'$ are $\sigma$-transcendental over $E$. Thus there is a difference field isomorphism $\psi : E\langle a\rangle\to E\langle a'\rangle$ over $E$ sending $a$ to $a'$. But, since $v(a) = \gamma$, $v(a') = \gamma'$, $\Gamma_E\langle\gamma\rangle\cong\Gamma_E\langle\gamma'\rangle$ over $\Gamma_E$, and $a$ and $a'$ are both generic over $\E$, the valuations are already determined and matched up on both sides, i.e., $\psi$ is actually a valued difference field isomorphism.
\end{proof}

\bigskip
\section{Embedding Theorem}
To prove completeness and relative quantifier elimination of the theory of multiplicative valued difference fields, we would use the following test:
\begin{test}
Let $T$ be a a theory in a first order language $\L$. Suppose that $T$ has no finite models, that $T$ is complete with respect to the atomic theory (i.e., for each atomic sentence $\psi$ of $\L$, $T\vdash \psi$ or $T\vdash \neg\psi$), and that $T$ has at least one constant symbol. Then the following are equivalent:
\begin{itemize}
\item $T$ is complete and eliminates quantifiers.
\item In any model of set theory in which GCH holds, if $\M, \mN \models T$ are saturated models of $T$ of the same cardinality, $A\subseteq \M$ and $B\subseteq \mN$ are substructures of cardinality strictly less than that of $\M$, and $f: A\to B$ is an $\L$-isomorphism, then there is an $\L$-isomorphism $g: \M\to \mN$ such that $g|_A = f$.
\item If $\M, \mN \models T$ are $\kappa^+$-saturated for some infinite cardinal $\kappa \ge |T|$, $A\subseteq \M$ is a substructure of $\M$ of cardinality at most $\kappa$, $f: A\hookrightarrow \mN$ is an $\L$-embedding, and $a\in \M$, then there is a substructure $A'$ of $\M$ containing $A$ and $a$ and an extension of $f$ to an $\L$-embedding $g: A'\hookrightarrow \mN$.
\end{itemize}
\end{test}
To that end we would like to know when can we extend isomorphism between ``small'' substructures, and the main theorem of this section, Theorem 9.5, gives an answer to that question.

For the moment we will work in a 4-sorted language, where we have our usual 3 sorts for the valued field $K$, the value group $\Gamma$ and the residue field $k$, and we add to it a fourth sort called the $RV$. This represents the language of the leading terms introduced in \cite{BK}, and explained further in \cite{S2} and \cite{F}. We could have just worked with a 2-sorted language with $K$ and $RV$ (we call this the leading term language). But the 2-sorted language is interpretable in and also interprets the 4-sorted language. So to make things more transparent we stick to the 4-sorted language. Before we proceed further, we would like to recall some preliminaries about the RV structures. Recall that we are always dealing with the equi-characteristic zero case.

\vspace{2em}
\hspace{-1.5em} \textbf{\underline{Preliminaries}} 

Let $\K = (K, \Gamma, k; v, \pi, \rho)$ be a multiplicative valued field. Let $\O$ be the ring of integers, and $\m$ be its maximal ideal. Let $K^\times$, $\O^\times$ and $k^\times$ denote the set of units of $K$, $\O$ and $k$ respectively. As it turns out $(1 + \m)$ is a subgroup of $K^\times$ under multiplication. We denote the factor group as $RV := K^\times/1 + \m$ and the natural quotient map as $\rv: K^\times\to RV$. To extend the map to whole of $K$, we introduce a new symbol ``$\infty$'' (as we do with value groups) and define $\rv(0) = \infty$. Though $RV$ is defined merely as a group, it inherits much more structure from $\K$.

To start with, since the valuation $v$ on $\K$ is given by the exact sequence
$$\xymatrix{1\ar[r] & \O^\times\ar[r]^\iota & K^\times \ar[r]^v & \Gamma\ar[r] & 0}$$
and since $1 + \m\le \O^\times$, the valuation descends to $RV$ giving the following exact sequence
$$\xymatrix{1\ar[r] & k^\times\ar[r]^\iota & RV \ar[r]^v & \Gamma\ar[r] & 0}$$
(note that $\O^\times/1 + \m = (\O/\m)^\times = k^\times$). In fact, it follows straight from the definitions that,
\begin{lem}
For all non-zero $x, y\in K$, the following are equivalent:
\begin{enumerate}
\item $\rv(x) = \rv(y)$
\item $v(x - y) > v(y)$
\item $\pi(x/y) = 1$
\end{enumerate}
\end{lem}
\begin{proof}
See \cite{F}, Proposition 1.3.3.

Also note that if $x, y\in\O$, then the last condition is equivalent to saying $\pi(x) = \pi(y)$. And both (2) and (3) imply that $v(x) = v(y)$.
\end{proof}

Since $\sigma$ is multiplicative (and hence $\sigma(\m) = \m$), the difference operator on $K$ induces a difference operator (which we again call by $\sigma$) on $RV$ as follows:
$$\sigma(\rv(x)) := \rv(\sigma(x)).$$
It trivially follows that the induced $\sigma$ is also multiplicative with the same $\rho$.

$RV$ also inherits an image of addition from $K$ via the relation
$$\oplus(x^1_{\rv}, \ldots, x^n_{\rv}, z_{\rv}) = \exists x^1, \ldots, x^n, z\in K\;(x^1_{\rv} = \rv(x^1)\; \wedge\;\cdots\;\wedge \; x^n_{\rv} = \rv(x^n)\; \wedge\; z_{\rv} = \rv(z) \;\wedge\; x^1 + \cdots + x^n = z).$$
The sum $x^1_{\rv} + \cdots + x^n_{\rv}$ is said to be well-defined (and $= z_{\rv}$) if there is exactly one $z_{\rv}$ such that $\oplus(x^1_{\rv}, \ldots, x^n_{\rv}, z_{\rv})$ holds. Unfortunately this is not always the case. Fortunately, the sum is well-defined when it is ``expected'' to be. Formally,
\begin{lem}
$\rv(x_1) + \cdots + \rv(x_n)$ is well-defined $($and is equal to $\rv(x_1 + \cdots + x_n))$ if and only if $v(x_1 + \cdots + x_n) = \min\{v(x_1), \ldots, v(x_n)\}$.
\end{lem}
\begin{proof}
See \cite{F}, Proposition 1.3.6, 1.3.7 and 1.3.8.
\end{proof}
Thus, we construe $RV$ as a structure in the language $\L_\rv := \{\cdot, ^{-1}, \oplus, v, \sigma, 1\}$. And finally, we have
\begin{prop}
$\Gamma$ and $k$ are interpretable in $RV$.
\end{prop}
\begin{proof}
See \cite{F}, Proposition 3.1.4. Note that the proof there is done for valued fields. For our case, the difference operator on $V$ is interpreted in terms of the difference operator on $RV$ as $\sigma(v(x)) = v(\sigma(x))$. And since the nonzero elements $\bar{x}$ of the residue field are in bijection with $x\in RV$ such that $v(x) = 0$, $\bar{\sigma}(\bar{x})$ is interpreted in the obvious way as $\sigma(x)$.
\end{proof}

Now we describe the embeddings. Let $\K = (K, \Gamma, k, RV; v, \pi, \rv, \rho)$ and $\K' = (K', \Gamma', k', RV'; v', \pi', \rv', \rho)$ be two $\sigma$-henselian multiplicative valued difference fields of equal characteristic zero, satisfying Axiom 1 (with the given $\rho$). Recall that, by Lemma 5.10, $\K$ and $\K'$ satisfy Axiom 2 and Axiom 3. We denote the difference operator of both $\K$ and $\K'$ by $\sigma$, and their ring of integers by $\O$ and $\O'$ respectively. Let $\E = (E, \Gamma_E, k_E, RV_E; v, \pi, \rv, \rho)$ and $\E' = (E', \Gamma_{E'}, k_{E'}, RV_{E'}; v', \pi', \rv', \rho)$ be valued difference subfields of $\K$ and $\K'$ respectively. We say a bijection $\psi: E\to E'$ is an \textit{admissible isomorphism} if it has the following properties:
\begin{enumerate}
\item $\psi$ is an isomorphism of multiplicative valued difference fields;
\item the induced isomorphism $\psi_\rv: RV_E\to RV_{E'}$ in the language $\L_\rv$ is \textit{elementary}, i.e., for all formulas $\varphi(x_1, \ldots, x_n)$ in $\L_\rv$, and $\xi_1, \ldots, \xi_n\in RV_E$,
$$RV\models\varphi(\xi_1,\ldots, \xi_n) \iff RV'\models\varphi(\psi_\rv(\xi_1), \ldots, \psi_\rv(\xi_n));$$
\item the induced isomorphism $\psi_r: k_E\to k_{E'}$ of difference fields is \textit{elementary}, i.e., for all formulas $\varphi(x_1, \ldots, x_n)$ in the language of difference fields, and $\alpha_1, \ldots, \alpha_n\in k_E$,
$$k\models\varphi(\alpha_1,\ldots, \alpha_n) \iff k'\models\varphi(\psi_r(\alpha_1), \ldots, \psi_r(\alpha_n));$$
\item the induced isomorphism $\psi_v: \Gamma_E \to \Gamma_{E'}$ is elementary, i.e, for all formulas $\varphi(x_1, \ldots, x_n)$ in the language of MODAG, and $\gamma_1, \ldots, \gamma_n\in\Gamma_E$,
$$\Gamma\models\varphi(\gamma_1,\ldots, \gamma_n) \iff \Gamma'\models\varphi(\psi_v(\gamma_1), \ldots, \psi_v(\gamma_n)).$$
\end{enumerate}
(Note that it is enough to maintain (1) and (2) above, since (3) and (4) are consequences of (2) because of Proposition 9.3).

Our main goal is to be able to extend such admissible isomorphisms. For this we need certain degree of saturation on $\K$ and $\K'$. Fix an infinite cardinal $\kappa$ and lets assume that $\K$ and $\K'$ are $\kappa^+$-saturated. Recall that, since the language is countable, such models exist assuming GCH. We then say a substructure $\E = (E, \Gamma_E, k_E, RV_E; v, \pi, \rv, \rho)$ of $\K$ (respectively of $\K'$) is \textit{small} if $|\Gamma_E|, |k_E| \le \kappa$. While extending the isomorphism, we do it in steps and at each step we typically extend the isomorphism from some $E$ to $E\langle a\rangle$, which is obviously small if $E$ is; and then reiterate the process $\kappa$ many times, which again preserves smallness. Eventually we reiterate this process countably many times and take union of an increasing sequence of countably many small fields, which also preserves smallness. Having said all that, we now state the Embedding Theorem.

\begin{thm}[Embedding Theorem]
Suppose $\K, \K', \E, \E'$ are as above with $\K, \K'$ $\kappa^+$-saturated and $\E, \E'$ small. Assume $\psi: E\to E'$ is an admissible isomorphism and let $a\in K$. Then there exist $b\in K'$ and an admissible isomorphism $\psi': E\langle a\rangle\cong E'\langle b\rangle$ extending $\psi$ with $\psi(a) = b$.
\end{thm}
\begin{proof}
Note that the theorem is obvious if $\Gamma = \{0\}$. So let's assume that $\Gamma\not=\{0\}$. Also wlog, we may assume $a\in\O_K$. We will extend the isomorphism in steps. Note that we have three cases to consider:
\begin{enumerate}
\item[Case(1).] There exists $c\in E\langle a\rangle$ such that $\pi(c)\in k\setminus k_E$;
\item[Case(2).] There exists $c\in E\langle a\rangle$ such that $v(c)\in\Gamma\setminus\Gamma_E;$
\item[Case(3).] For all $c\in E\langle a\rangle, \pi(c)\in k_E$ and $v(c)\in\Gamma_E$.
\end{enumerate}

\vspace{1em}
\hspace{-1.2em}\textbf{\underline{Step I: Extending the residue field}}

Let $c\in E\langle a\rangle$ be such that $\alpha := \pi(c)\not\in k_E$. Since $k^\times\hookrightarrow RV$, $\alpha\in RV$. By saturation of $\K'$, we can find $\alpha'\in RV'$ and an $\L_\rv$-isomorphism $RV_E\langle\alpha\rangle\cong RV_{E'}\langle\alpha'\rangle$ extending $\psi_\rv$ and sending $r\mapsto r'$ that is elementary as a partial map between $RV$ and $RV'$. Note that then $\alpha'\in k'$. Now we have two cases to consider.

\textit{Subcase I.} $\alpha$ (respectively, $\alpha'$) is $\bar{\sigma}$-transcendental over $k_E$ (respectively, $k_{E'}$). In that case, pick $d\in\O$ and $d'\in\O'$ such that $\pi(d) = \alpha$ and $\pi(d') = \alpha'$. Then by Lemma 8.1, there is an admissible isomorphism $\E\langle d\rangle\cong\E'\langle d'\rangle$ extending $\psi$ with small domain $(E\langle d\rangle, \Gamma_E, k_E\langle\alpha\rangle)$ and sending $d$ to $d'$.

\textit{Subcase II.} $\alpha$ is $\bar{\sigma}$-algebraic over $k_E$. Let $P(x)$ be a $\sigma$-polynomial over $\O_\E$ such that $\bar{P}(x)$ is a minimal $\bar{\sigma}$-polynomial of $\alpha$. Pick $d\in\O$ such that $\pi(d) = \alpha$. If $P(d)\not= 0$ already, then $(P, d)$ is in $\sigma$-hensel configuration with $\gamma(P, d) > 0$. Since $\K$ is $\sigma$-henselian, there is $e\in\O$ such that $P(e) = 0$ and $\pi(e) = \pi(d) = \alpha$. Likewise, there is $e'\in\O'$ such that $P^{\psi}(e') = 0$ and $\pi(e') = \alpha'$, where $P^{\psi}$ is the difference polynomial over $E'$ corresponding to $P$ under $\psi$. Then by Lemma 8.2, there is an admissible isomorphism $\E\langle e\rangle\cong\E'\langle e'\rangle$ extending $\psi$ with small domain $(E\langle e\rangle, \Gamma_E, k_E\langle\alpha\rangle)$ and sending $e$ to $e'$.

Note that in either case, we have been able to extend the admissible isomorphism to a small domain that includes $\alpha$. Since $E$ is small, so is $E\langle a\rangle$, i.e., $|k_{E\langle a\rangle}|\le\kappa$. Thus, by repeating Step I $\kappa$ many times, we are able to extend the admissible isomorphism to a small domain $\E_1$ such that for all $c\in E\langle a\rangle$ with $\pi(c)\not\in k_E$, we have $\pi(c)\in k_{E_1}$. Continuing this countably many times, we are able to build an increasing sequence of small domains $\E = \E_0\subset \E_1 \subset\cdots\subset\E_i\subset\cdots$ such that for each $c\in E_i\langle a\rangle$ with $\pi(c)\not\in k_{E_i}$, we have $\pi(c)\in k_{E_{i + 1}}$. Taking the union of these countably many small domains, we get a small domain, which we still call $\E$, such that $\psi$ extends to an admissible isomorphism with domain $E$ and for all $c\in E\langle a\rangle$, we have $\pi(c)\in k_E$, i.e., we are not in Case(1) anymore.\newline

\hspace{-1.2em}\textbf{\underline{Step II: Extending the value group}}

Let $c\in E\langle a\rangle$ be such that $\gamma := v(c)\not\in \Gamma_E$. Let $b\in K$ be generic over $\E$ with $v(b) = \gamma$. Let $r := \rv(b)$. By saturation of $\K'$, find $r'\in RV'$ and an $\L_\rv$-isomorphism $RV_E\langle r\rangle\cong RV'\langle r'\rangle$ extending $\psi_\rv$, sending $r\mapsto r'$, that is elementary as a partial map between $RV$ and $RV'$. Let $b' \in K'$ be such that $\rv'(b') = r'$. 

We claim that $b'$ is generic over $\E'$ : for any $P(x) = \sum b_{\vec{L}}\vec{\sigma}(x)^{\vec{L}}$ with $b_{\vec{L}}\in E'$, let $P^{\psi^{-1}}(x) = \sum a_{\vec{L}}\vec{\sigma}(x)^{\vec{L}}$ be the corresponding $\sigma$-polynomial over $E$ with $a_{\vec{L}}\in E$ and $\psi(a_{\vec{L}}) = b_{\vec{L}}$. Since $b$ is generic over $\E$, we have $v(P^{\psi^{-1}}(b)) = \min\{v(a_{\vec{L}}) + |\vec{L}|_\rho\cdot\gamma\}$, and hence by Lemma 9.3, we have
$$\rv(P^{\psi^{-1}}(b)) = \sum \rv(a_{\vec{L}}\vec{\sigma}(b)^{\vec{L}}) = \sum \rv(a_{\vec{L}})\vec{\sigma}(\rv(b))^{\vec{L}} = \sum \rv(a_{\vec{L}}) \vec{\sigma}(r)^{\vec{L}}.$$
Then 
$$\rv'(P(b')) = \psi_\rv(\rv(P^{\psi^{-1}}(b))) = \psi_\rv\Big(\sum \rv(a_{\vec{L}}) \vec{\sigma}(r)^{\vec{L}}\Big) = \sum \rv'(b_{\vec{L}}) \vec{\sigma}(r')^{\vec{L}} = \sum \rv'(b_{\vec{L}}\vec{\sigma}(b')^{\vec{L}}),$$
and hence by Lemma 9.3 again, we have $v(P(b')) = \min\{v(b_{\vec{L}}) + |\vec{L}|_\rho\cdot v(b')\}$, i.e., $b'$ is generic over $\E'$. Then by Lemma 8.8, since $\K$ satisfies Axiom 2, there is an admissible isomorphism from $\xymatrix{\E\langle b\rangle\ar^\cong[r] & \E'\langle b'\rangle}$ extending $\psi$ and sending $b\mapsto b'$.

Thus we have been able to extend the admissible isomorphism to a small domain that includes $\gamma$. Since $E$ is small, so is $E\langle a\rangle$, i.e., $|\Gamma_{E\langle a\rangle}|\le\kappa$. Thus, by repeating Step II $\kappa$ many times, we are able to extend the admissible isomorphism to a small domain $\E_1$ such that for all $c\in E\langle a\rangle$ with $v(c)\not\in\Gamma_E$, we have $v(c)\in\Gamma_{E_1}$. Continuing this countably many times, we are able to build an increasing sequence of small domains $\E = \E_0\subset \E_1 \subset\cdots\subset\E_i\subset\cdots$ such that for each $c\in E_i\langle a\rangle$ with $v(c)\not\in\Gamma_{E_i}$, we have $v(c)\in\Gamma_{E_{i + 1}}$. Taking the union of these countably many small domains, we get a small domain, which we still call $\E$, such that $\psi$ extends to an admissible isomorphism with domain $E$ and for all $c\in E\langle a\rangle$, we have $v(c)\in\Gamma_E$, i.e., we are not in Case(2) anymore.\newline

\hspace{-1.2em}\textbf{\underline{Step III: Immediate Extension}}

After doing Steps I and II, we are reduced to the case when $E\langle a\rangle$ is an immediate extension of $\E$ where both fields are equipped with the valuation induced by $\K$. Let $\E\langle a\rangle$ be the valued difference subfield of $\K$ that has $E\langle a\rangle$ as the underlying field. In this situation, we would like to extend the admissible isomorphism, not just to $\E\langle a\rangle$, but to a maximal immediate extension of $\E\langle a\rangle$ and use Corollary 6.10. However, for that we need $\E$ to satisfy Axiom 2 and Axiom 3. Since Axiom 3 implies Axiom 2 by Lemma 5.10, it is enough to extend $\E$ such that it satisfies Axiom 3. Recall that $\K$ satisfies Axiom 3. Now to make $\E$ satisfy Axiom 3, for each linear $\bar{\sigma}$-polynomial $P(x)$ over $k_E$, if there is already no solution to $P(x)$ in $k_E$, find a solution $\alpha\in k$ and follow Step I. Since there are at most $\kappa$ many such polynomials, we end up in a small domain. Thus, after doing all these, we can assume $\E$ satisfies Axiom 2 and Axiom 3. Let $\E^*$ be a maximal immediate valued difference field extension of $\E\langle a\rangle$. Then $\E^*$ is a maximal immediate extension of $\E$ as well. Similarly let $\E'^*$ be a maximal immediate extension of $\E'$. Since such extensions are unique by Corollary 6.8, and by Corollary 6.10 they can be embedded in $\K$ (respectively $\K'$) over $\E$ (respectively $\E'$) by saturatedness of $\K$ (respectively $\K'$), we have that $\psi$ extends to a valued field isomorphism $\E^*\cong\E'^*$. Since $k_{E^*} = k_E$ and $\Gamma_{E^*} = \Gamma_E$, it follows by Snake Lemma on the following diagram that $RV_{E^*} = RV_E$: 
$$\xymatrix{1\ar[r] & k^\times_E\ar[r]\ar[d]^{id} & RV_E\ar[r]\ar[d] & \Gamma_E\ar[r]\ar[d]^{id} & 0\\
1\ar[r] & k^\times_{E^*}\ar[r] & RV_{E^*}\ar[r] & \Gamma_{E^*}\ar[r] & 0}.$$
Thus, the isomorphism is actually admissible. It remains to note that $a$ is in the underlying difference field of $\E^*$.
\end{proof}

\bigskip
\section{Completeness and Quantifier Elimination Relative to RV}
We now state some model-theoretic consequences of Theorem 9.5. We use `$\equiv$' to denote the relation of elementary equivalence, and $\preccurlyeq$ to denote the relation of elementary submodel. Recall that we are working in the 4-sorted language $\L_4$ with sorts $K$ for the valued field, $\Gamma$ for the value group, $k$ for the residue field and $RV$ for the RV. The language also has a function symbol $\sigma$ going from the field sort to itself. Let $\K = (K, \Gamma, k, RV; v, \pi, \rv, \rho)$ and $\K' = (K', \Gamma', k', RV'; v', \pi',$ $\rv', \rho)$ be two $\sigma$-henselian multiplicative valued difference fields (in the 4-sorted language) of equi-characteristic zero satisfying Axiom 1 (with the same $\rho$).

\begin{thm}
$\K\equiv_{\L_4} \K'$ if and only if $RV\equiv_{\L_\rv} RV'$.
\end{thm}
\begin{proof}
The ``only if'' direction is obvious. For the converse, note that $(\Q, \{0\}, \Q, \Q;$ $v, \pi, \rv, \rho)$, with $v(q) = 0$, $\pi(q) = q$ and $\rv(q) = q$ for all $q\in\Q$, is a substructure of both $\K$ and $\K'$, and thus the identity map between these two substructures is an admissible isomorphism. Now apply Theorem 9.5.
\end{proof}
\begin{thm}
Let $\E = (E, \Gamma_E, k_E, RV_E; v, \pi, \rv, \rho)$ be a $\sigma$-henselian multiplicative valued difference subfield of $\K$, satisfying Axiom 1, such that $RV_E\preccurlyeq_{\L_\rv} RV$. Then $\E\preccurlyeq_{\L_4}\K$. 
\end{thm}
\begin{proof}
Take an elementary extension $\K'$ of $\E$. Then $\K'$ satisfies Axiom 1, and is also $\sigma$-henselian. Moreover $(E, \Gamma_E, k_E, RV_E; \cdots)$ is a substructure of both $\K$ and $\K'$, and hence the identity map is an admissible isomorphism. Hence, by Theorem 9.5, we have $\K\equiv_\E\K'$. Since $\E\preccurlyeq_{\L_4} \K'$, this gives $\E\preccurlyeq_{\L_4}\K$.
\end{proof}
\begin{cor}
$\K$ is decidable if and only if $RV$ is decidable.
\end{cor}

The proofs of these results use only weak forms of the Embedding Theorem, but now we turn to a result that uses its full strength: a relative elimination of quantifiers for the theory of $\sigma$-henselian multiplicative valued difference fields of equi-characteristic 0 that satisfy Axiom 1.

\begin{thm}
Let $T$ be the $\L_4$-theory of $\sigma$-henselian multiplicative valued difference fields of equi-characteristic zero satisfying Axiom 1, and $\phi(x)$ be an $\L_4$-formula. Then there is an $\L_4$-formula $\varphi(x)$ in which all occurrences of field variables are free, such that
$$T \vdash \phi(x) \iff \varphi(x).$$
\end{thm}
\begin{proof}
Let $\varphi$ range over $\L_4$-formulas in which all occurrences of field variables are free. For a model $\K = (K, \Gamma, k, RV; v, \pi, \rv, \rho)$ of $T$ and $a\in K^l$, $\gamma\in\Gamma^m$, $\alpha\in k^n$ and $r\in RV^s$, let
$$rqftp^{\K}(a, \gamma, \alpha, r) := \{\varphi : \K \models \varphi(a, \gamma, \alpha, r)\}.$$
Let $\K, \K'$ be models of $T$ and suppose
$$(a, \gamma, \alpha, r)\in K^l\times \Gamma^m\times k^n\times RV^s, \;\;\;\;\;\;\;\;\;\;(a', \gamma', \alpha', r')\in K'^l\times \Gamma'^m\times k'^n\times RV'^s$$
are such that $rqftp^{\K}(a, \gamma, \alpha, r) = rqftp^{\K'}(a', \gamma', \alpha', r')$. It suffices to show that
$$tp^{\K}(a, \gamma, \alpha, r) = tp^{\K'}(a', \gamma', \alpha', r').$$
Let $\E$ (respectively $\E'$) be the multiplicative valued difference subfield of $\K$ (respectively $\K'$) generated by $a, \gamma, \alpha$ and $r$ (respectively $a', \gamma', \alpha'$ and $r'$). Then there is an admissible isomorphism $\E\to\E'$ that maps $a\to a'$, $\gamma\to\gamma'$, $\alpha\to\alpha'$ and $r\to r'$. Now apply Theorem 9.5.
\end{proof}

\bigskip
\section{Completeness and Quantifier Elimination Relative to $(k, \Gamma)$}
Although the leading term language is already interpretable in the language of pure valued fields and is therefore closer to the basic language, the definable sets in this language are really obscure. We therefore would like to move to the 3-sorted language with the valued field $K$, value group $\Gamma$ and residue field $k$, where definable sets are much more transparent. It is well-known that in the presence of a ``cross-section'', the two-sorted structure $(K, RV)$ is interpretable in the three-sorted structure $(K, \Gamma, k)$. As a result any admissible isomorphism, as defined in the section on Embedding Theorem, boils down to one that satisfies properties (1), (3) and (4) only, because in the presence of a cross-section, (2) follows from (3) and (4). What that effectively means is that now we have completeness and quantifier elimination relative to the value group and the residue field. Before we can make all these explicit, we need to know when multiplicative valued difference fields can be equipped with a cross-section. So let $\K=(K, \Gamma, k; v, \pi, \rho)$ be a multiplicative valued difference field.
\begin{defn}
A cross-section $s:\Gamma\to K^\times$ on $\K$ is a cross-section on $\K$ as valued field such that for all $\gamma\in\Gamma$ and $\tau = \sum_{j = 0}^{n} i_j\sigma^j\in\Z[\sigma]$,
$$s(\tau(\gamma)) = \vec{\sigma}(s(\gamma))^{\vec{I}},$$
where $\vec{I} = (i_0, \ldots, i_n)$. As an example, for Hahn difference fields $k((t^\Gamma))$ we have a cross-section given by $s(\gamma) = t^\gamma$.
\end{defn}
Recall that we construe $K^\times$ as a left $\Z[\sigma]$-module (w.r.t. multiplication) under the action 
$$\Big(\sum_{j = 0}^n i_j\sigma^j\Big)a = \vec{\sigma}(a)^{\vec{I}},$$
where $\vec{I} = (i_0, \ldots, i_n)$ (we will freely switch between these two notations and the corresponding $\vec{I}$ or the $i_j$'s will be clear from the context); similarly we construe $\Gamma$ also as a left $\Z[\sigma]$-module (w.r.t. addition) under the action
$$\Big(\sum_{j = 0}^n i_j\sigma^j\Big)\gamma = \sum_{j = 0}^n i_j\rho^j\cdot\gamma.$$
Also, since $v(\vec{\sigma}(a)^{\vec{I}}) = \sum_{j = 0}^n i_j\rho^j\cdot v(a)$, we have an exact sequence of $\Z[\sigma]$-modules
$$\xymatrix{1 \ar[r] & \O^\times \ar[r]^\iota & K^\times \ar[r]^v & \Gamma \ar[r] & 0},$$
where $\O^\times$ is multiplicative group of units of the valuation ring $\O$. Clearly then, existence of a cross-section on $\K$ corresponds to this exact sequence being a split sequence. Before we proceed further, we need a few preliminaries from algebra.

\vspace{2em}
\hspace{-1.2em}\textbf{\underline{Preliminaries.}} Let $R$ be a commutative ring with identity (for our case $R = \Z[\sigma])$.
\begin{defn}
For two left $R$-modules $N\subseteq M$, $N$ is said to be pure in $M$ (notation: \xymatrix{N\ar@{^(->}^<<{p}[r] & M}) if for any $m-by-n$ matrix $(r_{ij})$ with entries in $R$, and any set $y_1, \ldots, y_m$ of elements of $N$, if there exist elements $x_1, \ldots, x_n\in M$ such that
$$\sum_{j = 1}^n r_{ij} x_j = y_i \;\;\;\;\;\;\;\;\mbox{for } i = 1,\ldots, m$$
then there also exist elements $x'_1, \ldots, x'_n\in N$ such that
$$\sum_{j = 1}^n r_{ij} x'_j = y_i \;\;\;\;\;\;\;\;\mbox{for } i = 1,\ldots, m.$$
\end{defn}
\begin{defn}
If $M, N$ are left $R$-modules and $f: N\to M$ is an injective homomorphism of left $R$-modules, then $f$ is called pure injective if $f(N)$ is pure in $M$ (notation: \xymatrix{N\ar@{^(->}^<<{p}^f[r] & M}).
\end{defn}
\begin{defn}
A left $R$-module $E$ is called pure-injective if for any pure injective module homomorphism $f: X \to Y$, and an arbitrary module homomorphism $g: X\to E$, there exist a module homomorphism $h: Y\to E$ such that $hf = g$, i.e. the following diagram commutes:
\begin{eqnarray*}
\xymatrix{X\ar@{^(->}^<<{p}^f[rr] \ar[d]^g & & Y\ar@{-->}^h[dll] \\
E & &}
\end{eqnarray*}
\end{defn}
\begin{thm}
Every $|R|^+$-saturated left $R$-module $E$ is pure-injective.
\end{thm}
\begin{proof}
See \cite{Ch}, page 171.
\end{proof}

That's all the preliminaries we need. Since $\Z[\sigma]$ is countable, any $\aleph_1$-saturated $\Z[\sigma]$-module is pure injective by Theorem 10.5. So let's assume $\K$ is $\aleph_1$-saturated. Then $\O^\times$ is also $\aleph_1$-saturated (as it is definable in $\K$) and hence pure-injective. Now since we have the exact sequence
$$\xymatrix{1 \ar[r] & \O^\times \ar[r]^\iota & K^\times \ar[r]^v & \Gamma \ar[r] & 0},$$
if we can show that $\O^\times$ is pure in $K^\times$, then we will be able to complete the following diagram
$$\xymatrix{\O^\times\ar@{^(->}^<<{p}^\iota[rr] \ar[d]^{\mbox{id}} & & K^\times\ar@{-->}^h[dll] \\ \O^\times & &}$$
which will give us a splitting of the above exact sequence.

Unfortunately $\O^\times$ is not pure in $K^\times$ in general. As it turns out, if $\rho$ is transcendental, then $\O^\times$ is pure in $K^\times$ without any further assumption on $\K$. However, if $\rho$ is algebraic, we need a further axiom to have the cross-section exist.

\vspace{2em}
\hspace{-1.2em}\textbf{\underline{Case I : $\rho$ is transcendental}}

In this case, $\Gamma$ is torsion free as a $\Z[\sigma]$-module, as $\sum_{j = 0}^n i_j \rho^j \not= 0$ for any tuple $\vec{I} = (i_0, \ldots, i_n)\in\Z^{n+1}$, $\vec{I}\not=\vec{0}$. Thus, for any $w\in \O^\times$ and $\vec{I}\in\Z^{n+1}$ with $\vec{I}\not= \vec{0}$, if $\vec{\sigma}(m)^{\vec{I}} = w$ for some $m\in K^\times$, then we have
$$\sum_{j = 0}^n i_j\rho^j\cdot v(m) = v(\vec{\sigma}(m)^{\vec{I}}) = v(w) = 0,$$
which implies $v(m) = 0$, i.e., $m\in \O^\times$. 

Now suppose we have a system of equations $\sum_{j = 1}^n (Q_{ij})x_j = y_i$, $i = 1, \ldots, m$, with $Q_{ij}$'s from $\Z[\sigma]$ and $y_i$'s from $\O^\times$, which has a solution $\bar{x} = (x_1, \ldots, x_n)^T$ in $K^\times$. Let $M = (Q_{ij})$ be the corresponding $m-by-n$ matrix and $\bar{\gamma} = (v(x_1), \ldots, v(x_n))^T$. Then $M\bar{\gamma} = \bar{0}$ is the corresponding $\Z[\sigma]$-linear system over $\Gamma$. If $m > n$, then by elementary linear algebra we know that at least $(m - n)$ equations are linearly dependent on the remaining $n$ equations, and doesn't give us any new information. So by row reduction, we may assume that $m\le n$. Now if $m < n$, then again by elementary linear algebra, we know that there are at most $m$ pivot columns and hence at least $(n - m)$ free columns (variables), which means we may assign any value to those variables and still have a solution for the whole system. Assigning the value $0$ to those free variables, we may assume $m = n$, i.e., it is a square system. If the determinant of $M$ is zero, then again one of the rows is linearly dependent on the other rows, and by row reduction we can reduce to one of the earlier cases. Thus, we may assume $M$ has non-zero determinant. But such a system has the unique solution $\bar{\gamma} = \bar{0}$, which implies the $x_i$'s are already in $\O^\times$.

Hence, $\O^\times$ is pure in $K^\times$. And so we have a cross-section $s: \Gamma\to K^\times$.

\vspace{2em}
\hspace{-1.2em}\textbf{\underline{Case II : $\rho$ is algebraic}}

Let $P(x) = i_0 + i_1 x + \cdots + i_n x^n$ be the minimal (monic) polynomial of $\rho$ over $\Z$. Let $\vec{I} = (i_0, \ldots, i_n)$ and define
$$P^\sigma := \sum_{j = 0}^n i_j \sigma^j.$$
Note that for any $a\in K^\times$, $v((P^\sigma)a) = v((\sum_{j = 0}^n i_j \sigma^j)a) = v(\vec{\sigma}(a)^{\vec{I}}) = P(\rho)\cdot v(a) = 0$. Clearly in this case, $\Gamma$ is not torsion-free. In fact, for any $\gamma\in\Gamma$, $Tor(\gamma) = (P^\sigma)$, which is a prime ideal in $\Z[\sigma]$. To make things work here, we need the following axiom:
$$\mbox{\textbf{Axiom 4.} }\forall\gamma\in\Gamma\;\exists a\in K^\times (v(a) = \gamma\; \wedge\; (P^\sigma)a = 1).$$
To check that this works, let us define
\begin{eqnarray*}
G_K & = & \{a\in K^\times : (P^\sigma)a = 1\} \;\;\mbox{ and}\\
G_{\O} & = & \{a\in G_K : v(a) = 0\}
\end{eqnarray*}
It is routine to check that $G_K$ is a subgroup of $K^\times$, and $G_{\O}$ is a subgroup of $\O^\times$. Since both of them are definable in $\K$, they are $\aleph_1$-saturated too, and hence pure-injective. Moreover by Axiom 4, $v|_{G_K} : G_K\to\Gamma$ is surjective, and so
$$\xymatrix{1\ar[r] & G_{\O}\ar^\iota[r] & G_K\ar[r]^v & \Gamma\ar[r] & 0}$$
is an exact sequence. We claim that $G_{\O}$ is pure in $G_K$.

Note that, for any $S\in\Z[\sigma]$, we can write $S = QP^\sigma + R$, for some $S, R\in\Z[\sigma]$ with $R = 0$ or deg$(R) < $ deg$(P^\sigma)$. Then, for any $a\in G_K$, we have
$$(S)a = (QP^\sigma + R)a = (Q)(P^\sigma)a \cdot (R)a = (Q)1 \cdot (R)a = (R)a,$$
i.e., $G_K$ as a $\Z[\sigma]$-module is isomorphic to $G_K$ as a $\Z[\sigma]/\langle P^\sigma\rangle$-module. Similarly for $G_{\O}$. Moreover it follows that if $R\not= 0$, i.e., if $P^\sigma$ does not divide $S$, then
$$(S)v(a) = 0 \implies (R)v(a) = 0 \implies v(a) = 0,$$
the last equality follows from the fact that deg$(R) < $ deg$(P^\sigma)$ and $P$ is the minimal polynomial for $\rho$. Thus, $\Gamma$ is torsion-free as a $\Z[\sigma]/\langle P^\sigma\rangle$-module. Therefore, the sequence
$$\xymatrix{1\ar[r] & G_{\O}\ar^\iota[r] & G_K\ar[r]^v & \Gamma\ar[r] & 0}$$
is exact as a map of $\Z[\sigma]/\langle P^\sigma\rangle$-modules. As $\Gamma$ is now torsion-free as a $\Z[\sigma]/\langle P^\sigma\rangle$-module, we can follow the same argument as in the transcendental case and have that $G_{\O}$ is pure in $G_K$ (recall that $\Z[\sigma]/\langle P^\sigma\rangle$ is an integral domain as $\langle P^\sigma\rangle$ is a prime ideal). In particular, there is a section $s:\Gamma\to G_K$ as a map of $\Z[\sigma]/\langle P^\sigma\rangle$-modules, which is easily seen to be a map of $\Z[\sigma]$-modules also, since $G_K$ as a $\Z[\sigma]$-module is isomorphic to $G_K$ as a $\Z[\sigma]/\langle P^\sigma\rangle$-module. Since $G_K$ is a subgroup of $K^\times$, we have our required section $s:\Gamma\to K^\times$. \newline

The upshot of all these is the following
\begin{thm}
Each multiplicative valued difference field $\K = (K, \Gamma, k; v, \pi, \rho)$ satisfying Axiom 4 has an elementary extension which can be equipped with a cross-section.
\end{thm}

\begin{rmk}
\begin{enumerate}
\item Note that if $\rho$ is an honest integer, say $n$, then $P(x) = x - n$ and hence $P^\sigma(x) = x^{-n}\sigma(x)$. Thus for $\rho = n$, Axiom 4 says : 
$$\forall\gamma\in\Gamma\;\exists a\in K^\times ( v(a) = \gamma \mbox{ and } \sigma(x) = x^n).$$ 
For $\rho = 1$, this is precisely the axiom of ``enough constants'' of \cite{S}. Following this analogy, we will say that $\K$ has enough constants if either $\rho$ is transcendental, or $\rho$ is algebraic with minimal polynomial $P(x)$ and $\K$ satisfies Axiom 4.
\item If $\rho = 1$ and $\K$ is $\sigma$-henselian, then $\K$ satisfies Axiom 4 automatically : Let $\gamma\in\Gamma$. WMA $\gamma \ge 0$. Let $c\in K$ be such that $v(c) = \gamma$. Consider 
$$Q(\epsilon) = \sigma(\epsilon) - \dfrac{c}{\sigma(c)}\epsilon.$$
Note that $v\Big(\dfrac{c}{\sigma(c)}\Big) = 0$. Hence, $Q$ is a linear $\sigma$-polynomial over $\O$. Thus, $\bar{Q}(\bar{\epsilon})$ is a linear $\bar{\sigma}$-polynomial over $k$. Since $k$ is linear difference closed (by Lemma 5.10), we can find $\bar{\epsilon}\in k$ such that $\bar{Q}(\bar{\epsilon}) = 0$. Choose $\epsilon\in\O$ such that $\pi(\epsilon) = \bar{\epsilon}$. In particular, $v(\epsilon) = 0$ and $(Q, \epsilon)$ is in $\sigma$-hensel configuration with $\gamma(Q, \epsilon) > 0$, and hence has a root, say $b$, with $v(b - \epsilon) = \gamma(Q, \epsilon) > 0$. This forces $v(b) = 0$ and
$$\sigma(cb) = \sigma(c)\sigma(b) = cb.$$
Note that $v(cb) = v(c) + v(b) = \gamma$. Set $a := cb$.
\item If $\rho = \frac{p}{q}\in\Q$, then one might follow the pattern of (2) and, instead of imposing Axiom 4 on $\K$, demand that the residue field $k$ is not only linear difference closed, but also satisfies equations of the form
$$Q(x) = \bar{\sigma}(x) - ax^{p/q},$$
where $a\in k$. Then Axiom 4 is automatically enforced on a $\sigma$-henselian $\K$. However, it should be noted that taking this approach is stronger than imposing Axiom 4, because Axiom 4 doesn't necessarily imply solutions to such equations in the residue field.
\item Axiom 4 is consistent with the other axioms. In particular, the formal power series fields described in Section 7 are a model of Axiom 4 too. For each $\gamma$, one special element from $k((t^\Gamma))$ satisfying Axiom 4 is $t^\gamma$. Also note that the map $s:\Gamma\to k((t^\Gamma))^\times$ sending $\gamma\mapsto t^\gamma$ is a cross-section on $k((t^\Gamma))$.
\end{enumerate}
\end{rmk}

Once we have the cross-section in place, we have the following result:
\begin{prop}
Suppose $\K$ has a cross-section $s:\Gamma\to K^\times$. Then $RV$ is interpretable in the two-sorted structure $(\Gamma, k)$ with the first sort in the language of $MODAG$ and the second in the language of difference fields. 
\end{prop}
\begin{proof}
Let $S = (\Gamma\times k^\times)\cup\{(0, 0)\}$. Note that $S$ is a definable subset of $\Gamma\times k$ (in particular, the second co-ordinate is zero only when the first is too). Define $f: S\to RV\cup\{\infty\}$ by
$$f((\gamma, a)) = \left\{ \begin{tabular}{ll} $s(\gamma)a$ & if $a\not= 0$\\ $\infty$ & if $(\gamma, a) = (0, 0)$\end{tabular}\right.$$
Now it follows from \cite{F}, Proposition 3.1.6, that $f$ is a bijection, and that the inverse images of multiplication and $\oplus$ on $RV$ are definable in $S$. Moreover, if $a\not= 0$, then $v(s(\gamma)a) = v(s(\gamma)) + v(a) = \gamma + 0 = \gamma$, and if $a = 0$, then $v(\infty) = \infty$. Thus the inverse image of the valuation map is $\{\langle (\gamma, a), \gamma\rangle\}\cup\{\langle (0, 0), \infty\rangle\}$. Finally, since $\sigma(s(\gamma)a) = s(\sigma(\gamma))\bar{\sigma}(a)$, the inverse image of the difference operator on $RV$ is given by $\{ \langle (\gamma, a), (\sigma(\gamma), \bar{\sigma}(a))\rangle\}$. Hence the result follows.
\end{proof}

As an immediate corollary of Proposition 11.8 and Theorem 11.6, we have
\begin{cor}
If $\K = (K, \Gamma, k, RV; v, \pi, \rv, \rho)$ and $\K' = (K', \Gamma', k', RV'; v', \pi', \rv', \rho)$ are two multiplicative valued fields satisfying Axiom 1 (with the same $\rho$) and Axiom 4, and $\Gamma\equiv\Gamma'$ in the language of $MODAG$ and $k\equiv k'$ in the language of difference fields, then $RV\equiv_{\L_\rv} RV'$.
\end{cor}

This allows us to work in the 3-sorted language $\L_3$, where we have a sort $K$ for the valued field, a sort $\Gamma$ for the value group and a sort $k$ for the residue field, eliminating the need for the $RV$ sort. Recall that the language also has a function symbol $\sigma$ going from the field sort to itself. Combining Corollary 11.9 with Theorems 10.1, 10.2 and 10.4 and Corollary 10.3, we then have the following nice results. Let $\K = (K, \Gamma, k; v, \pi, \rho)$ and $\K' = (K', \Gamma', k'; v', \pi', \rho)$ be two $\sigma$-henselian multiplicative valued difference fields, satisfying Axiom 1 (with the same $\rho$) and Axiom 4, of equi-characteristic zero. Then
\begin{thm}
$\K\equiv_{\L_3}\K'$ if and only if $\Gamma\equiv\Gamma'$ in the language of $MODAG$ and $k\equiv k'$ in the language of difference fields.
\end{thm}
\begin{thm}
Let $\E = (E, \Gamma_E, k_E; v, \pi, \rho)$ be a $\sigma$-henselian multiplicative valued difference subfield of $\K$, satisfying Axiom 1 and Axiom 4, such that $\Gamma_E\preccurlyeq\Gamma$ in the language of $MODAG$ and $k_E\preccurlyeq k$ in the language of difference fields. Then $\E\preccurlyeq_{\L_3}\K$.
\end{thm}
\begin{thm}
$\K$ is decidable if and only if $\Gamma$ and $k$ are decidable.
\end{thm}
And finally,
\begin{thm}
Let $T$ be the $\L_3$-theory of $\sigma$-henselian multiplicative valued difference fields of equi-characteristic zero satisfying Axiom 1 and Axiom 4, and $\phi(x)$ be an $\L_3$-formula. Then there is an $\L_3$-formula $\varphi(x)$ in which all occurrences of field variables are free, such that
$$T \vdash \phi(x) \iff \varphi(x).$$
\end{thm}

\end{document}